\documentclass[12pt,a4paper]{amsart}

\usepackage[top=35mm, bottom=35mm, left=30mm, right=30mm]{geometry}
\usepackage[colorlinks=true,citecolor=blue]{hyperref}
\usepackage{mathptmx}
\usepackage{eucal}
\usepackage{graphicx}
\usepackage{mathrsfs}
\usepackage{amssymb}
\usepackage{amsmath}
\usepackage{amsthm}
\usepackage{amscd}
\usepackage{xcolor}
\usepackage{verbatim}
\usepackage[pagewise]{lineno}
\usepackage{tikz-cd}
\usepackage{pgfplots}
\pgfplotsset{compat=newest}

\theoremstyle{plain}

\newtheorem{thm}{Theorem}[section]
\newtheorem{cor}[thm]{Corollary}
\newtheorem{lem}[thm]{Lemma}
\newtheorem{prop}[thm]{Proposition}

\newtheorem{ques}[thm]{Question}

\theoremstyle{definition}
\newtheorem{defn}[thm]{Definition}
\newtheorem{rem}[thm]{Remark}
\newtheorem{exmp}[thm]{Example}

\numberwithin{equation}{section}

\def\htop{h_{\rm top}}
\def\Ptop{P_{\rm top}}

\newcommand{\N}{\mathbb{N}}
\newcommand{\R}{\mathbb{R}}

\newcommand{\Z}{\mathbb{Z}}

\newcommand{\ep}{\epsilon}

\def \C {\mathcal C}

\numberwithin{equation}{section}

\input xy
\xyoption{all}
\begin{document}
\title[Variational principles of topological pressure for correspondences]
{Variational principles of topological pressure for correspondences} 

\author[T. Wang]{Tao Wang}
\address{T. Wang: LCSM (Ministry of Education), School of Mathematics and Statistics, Hunan Normal University, Changsha, Hunan 410081, P. R. China}
\email{twang@hunnu.edu.cn}



\thanks{This work was supported by Natural Science Foundation of Hunan Province (2023JJ40423)}

\maketitle

\begin{abstract}
Recently, Li, Li and Zhang introduced the topological pressure for correspondences and measure-theoretic entropy for transition probability kernels. Building thereon, they established a variational principle for correspondences satisfying the forward expansiveness condition.
In this work, we extend this research by deriving two types of variational principles:
\begin{enumerate}
  \item [(i)] For a class of correspondences, the topological pressure equals the supremum of the measure-theoretic pressures over extreme points of invariant measures.
  \item [(ii)] An abstract variational principle holds for general correspondences without requiring forward expansiveness.
\end{enumerate}
Furthermore, the differentiability and equilibrium states of the topological pressure for correspondences are also investigated.
\end{abstract}

\section{Introduction}

\subsection{Classical variational principle for single-valued continuous maps}

A {\it topological dynamical system} is a pair $(X,f)$ where $X$
is a compact metric space and $f:~X\rightarrow X$ is a continuous self-map.
Similarly, a {\it measure-preserving dynamical system} is a quadruple  $(X,\mathscr{M}(X),\mu,f)$ consisting of a set $X$, a $\sigma$-algebra $\mathscr{M}(X)$ on $X$, and a measure-preserving transformation $f$ on the probability space $(X,\mathscr{M}(X),\mu)$.
Let $\mathcal{P}(X), \mathcal{P}_f(X)$, $\mathcal{P}^e_f(X)$ denote the sets of all Borel probability measures, $f$-invariant Borel probability measures, and $f$-invariant ergodic Borel probability measures on $X$, respectively.
Let $(X,f)$ be a topological dynamical system and $\mu\in \mathcal{P}_f(X)$ be an $f$-invariant Borel probability measure, the system $(X,f)$ naturally induces a measure-preserving dynamical system $(X,\mathscr{B}(X),\mu,f)$, where $\mathscr{B}(X)$ refers to the Borel $\sigma$-algebra on $X$. For a real-valued continuous potential function $\varphi$ on $X$, define
$\Ptop(f,\varphi)$ as the topological pressure of $\varphi$, and
$P_\mu(f,\varphi):=h_\mu(f)+\int\varphi\,d\mu$ as the measure-theoretic pressure of $\varphi$ for $\mu$,
where $h_\mu(f)$ is the measure-theoretic entropy of $\mu$.
The classical variational principle for topological
pressure states that
\begin{equation}\label{eq:classical variational principle}
\Ptop(f,\varphi)=\sup_{\mu\in \mathcal{P}_f(X)}\left\{P_\mu(f,\varphi)\right\}=\sup_{\mu\in \mathcal{P}^e_f(X)}\left\{P_\mu(f,\varphi)\right\}.
\end{equation}
This variational principle was established by Ruelle \cite{Ru73} and  Walters \cite{Wa75}.
An $f$-invariant Borel probability measure that attains the supremum is called an equilibrium state for $f$ and $\varphi$. In particular, if the potential $\varphi\equiv0$, then the equilibrium state is called a measure of maximal entropy.
The variational principle for topological pressure establishes a fundamental connection between ergodic theory and dynamical systems, serving as a cornerstone in multifractal analysis and dimension theory in dynamical systems \cite{KH95,Ke98,Pe97}.

\subsection{Variational principle for correspondences}

A {\it correspondence} $T$ on a compact metric space $X$ is a map from $X$ to the set of all nonempty closed subsets of $X$, such that the graph $\{(x,y)\in X^2:y\in T(x)\}$ is closed in $X^2$. This structure is also termed upper semi-continuous set-valued functions in \cite{KT17}, set-valued maps in \cite{RT18}, and closed relations in \cite{MA99}. As a natural generalization of single-valued continuous maps, correspondences arise extensively in control theory \cite{Po21}, differential games \cite{Pe93}, mathematical economics and game theory \cite{CPMP08}, and among other fields.

Now we recall some fundamental advances in the dynamical systems theory of correspondences. Topologically, foundational contributions include the extension of Poincar\'{e}'s recurrence to correspondences by Aubin, Frankowska, and Lasota \cite{AFL91}. Meanwhile, the notion of topological entropy is also extended to correspondences from different perspectives by many authors \cite{AK21,CP16,KT17,WZZ23,ZZ24}, with significant contributions addressing its properties and estimation. Moreover, several variants of expansiveness and specification properties have been proposed and analyzed  in \cite{CP16,PV17,RT18}.
Measure-theoretically, invariant measures and their equivalent characterizations for correspondences have been systematically investigated in \cite{MA99} (see also \cite{AFL91,Mi95}).

Formulating a rigorous variational principle for the topological entropy of correspondences is highly significant yet poses substantial challenges.
Very recently, Li, Li and Zhang \cite{LLZ23} systematically developed a thermodynamic formalism for correspondences. They first introduced the definitions of topological pressure for correspondences and measure-theoretic entropy for transition probability kernels. Subsequently, they established a variational principle for correspondences satisfying the forward expansiveness condition. Furthermore, the authors constructed a thermodynamic formalism for equilibrium states of correspondences endowed with some strong expansion properties. Finally, these results were applied to holomorphic and anti-holomorphic correspondences. More precisely, Li, Li and Zhang \cite{LLZ23} proved the following variational principle:

\begin{thm}\label{thm:variational-principle-correspondence}\cite[Theorem A]{LLZ23}.
Let $(X, d)$ be a compact metric space, $T$ be a forward expansive correspondence on $X$, and $\phi: \mathcal{O}_2(T) \rightarrow \mathbb{R}$ be a continuous function. Then the variational principle holds:
\[
\Ptop(T, \phi) = \sup_{\mathcal{Q},\, \mu} \left\{ h_\mu (\mathcal{Q}) + \int_X \int_{T(x_1)} \phi (x_1, x_2) \, d\mathcal{Q}_{x_1}(x_2) \, d\mu (x_1) \right\},
\]
where the supremum is taken over all pairs $(\mathcal{Q},\mu)$ such that:
\begin{itemize}
  \item $\mathcal{Q}$ is a transition probability kernel on $X$ supported by $T$, and
  \item $\mu$ is a $\mathcal{Q}$-invariant Borel probability measures on $X$.
\end{itemize}
Furthermore, this supremum can be attained at some pair $(\mathcal{Q}, \mu)$.   \end{thm}

Indeed, the variational principle in Theorem \ref{thm:variational-principle-correspondence}
can be reformulated in terms of invariant measures for correspondences as follows:
\begin{equation}\label{eq:variational principle2 for correspondence}
\Ptop(T, \phi) = \sup_{\mu \in \mathcal{P}_T(X)} \left\{ P_\mu (T,\phi) \right\},
\end{equation}
where $\mathcal{P}_T(X)$ is the set of all $T$-invariant Borel probability measures on $X$ and $P_\mu (T,\phi)$ is the measure-theoretic pressure of $\phi$ for $\mu$ (see Theorem \ref{thm:restatement of variational principle for correspondence} for more details).

\subsection{Our work}
Inspired by the classical variational principle for topological pressure \eqref{eq:classical variational principle}, we aim to investigate whether the aforementioned variational principle \eqref{eq:variational principle2 for correspondence} remains valid when replacing the set \(\mathcal{P}_T(X)\) of \(T\)-invariant probability measures with \(\mathcal{P}^e_T(X)\), the collection of all extreme points of the compact convex set \(\mathcal{P}_T(X)\) (see Remark \ref{rem:remark of invariant measure for correspondences} for more details). That is, we pose the following fundamental problem:

\begin{ques}\label{ques:1}
Let $(X, d)$ be a compact metric space, $T$ be in a specified class of correspondences on $X$ and $\phi: \mathcal{O}_2(T) \rightarrow \mathbb{R}$ be a continuous function. Does the topological pressure satisfy
\[
\Ptop(T, \phi) = \sup_{\mu \in \mathcal{P}^e_T(X)} \left\{ P_\mu (T,\phi) \right\}?
\]
where $\mathcal{P}^e_T(X)$ is the set of all extreme points of the compact convex set \(\mathcal{P}_T(X)\).
\end{ques}

Since the classical variational principle for topological pressure \eqref{eq:classical variational principle} holds for general topological dynamical systems, a natural question arises:
Can a variational principle for topological pressure be established for general correspondences?

\begin{ques}\label{ques:2}
Let $(X, d)$ be a compact metric space, $T$ be a correspondences on $X$ and $\phi: \mathcal{O}_2(T) \rightarrow \mathbb{R}$ be a continuous function. How can we define an appropriate quantity $\mathfrak{h}_\mu(\mathcal{Q})$ such that
\[
\Ptop(T, \phi) = \sup_{\mathcal{Q},\, \mu} \left\{ \mathfrak{h}_\mu (\mathcal{Q}) + \int_X \int_{T(x_1)} \phi (x_1, x_2) \, d\mathcal{Q}_{x_1}(x_2) \, d\mu (x_1) \right\},
\]
where $\mathcal{Q}$ ranges over all transition probability kernels on $X$ supported by $T$, and $\mu$ ranges over all $\mathcal{Q}$-invariant Borel probability measures on $X$.
\end{ques}

In the present paper, we address the aforementioned questions by establishing two classes of variational principles for correspondences:
\begin{enumerate}
  \item [(i)] For a class of correspondences, the topological pressure satisfies
  \[
  \Ptop(T, \phi) = \sup_{\mu \in \mathcal{P}^e_T(X)} \left\{ P_\mu (T,\phi) \right\},
  \]
  where $\mathcal{P}^e_T(X)$ is the set of all extreme points of the compact convex set \(\mathcal{P}_T(X)\).
  \item [(ii)] An abstract variational principle holds for general correspondences without the forward expansiveness hypothesis. Specifically, we introduce a quantity $\mathfrak{h}_\mu(\mathcal{Q})$ such that
  \[
  \Ptop(T, \phi) = \max_{\mathcal{Q}, \mu} \left\{ \mathfrak{h}_\mu (\mathcal{Q}) + \int_X \int_{T(x_1)} \phi (x_1, x_2) \, d\mathcal{Q}_{x_1}(x_2) \, d\mu (x_1) \right\},
  \]
  where $\mathcal{Q}$ ranges over all transition probability kernels on $X$ supported by $T$, and $\mu$ ranges over all Borel probability measures or $\mathcal{Q}$-invariant Borel probability measures on $X$.
\end{enumerate}
Additionally, we investigate the differentiability of the topological pressure and characterize its equilibrium states.

\section{Preliminary}

This section reviews essential foundations for our analysis: basic notations, the definitions of correspondences and transition probability kernels, topological pressure for correspondences, and measure-theoretic entropy for transition probability kernels. All these definitions presented here are explicitly drawn from \cite{LLZ23}.

\subsection{Basic notations}

In this subsection, we introduce some basic notations to be used throughout this paper.

Let $\N= \{1, 2, 3, \dots\}$, $\N_0= \{0, 1, 2, 3, \dots\}$ and
$\hat{\N}= \N\cup\{\omega\}$. Here $\omega$ is the least infinite ordinal.
Let $X$ be a set and $n\in\N$. Define the reversal $\gamma_n:X^n\to X^n$ by
\[
\gamma_n(x_1,\ldots,x_n):=(x_n,\ldots,x_1) \text{ for all }(x_1,\ldots,x_n)\in X^n.
\]
Denote by $\mathscr{M}(X)$ a $\sigma$-algebra on $X$.


Let \(X\) be a compact metric space. We denote by
\begin{itemize}
  \item $\mathscr{B}(X)$ the (completed) Borel $\sigma$-algebra on $X$,
  \item $\mathcal{P}(X)$ the set of (completed) Borel probability measures on $X$,
  \item \(\mathcal{F}(X)\) the set of all non-empty closed subsets of \(X\), and
  \item $C(X)$ the space of real-valued continuous functions on $X$.
\end{itemize}

Let $X$ be a compact metric space with the metric $d$ and $T: X\to \mathcal{F}(X)$ be a map. For any $A\subset X$, define $T(A):=\bigcup_{x\in A}T(x)$. For $n\in\N$, define $T^{n}(A)$ inductively on $n$ with $T^1(A):=T(A)$ and $T^{n+1}(A):=T(T^n(A))$. Moreover, define $T^{-1}(A):=\{x\in X:~T(x)\cap A\neq\emptyset\}$. For $n\in\N$, define $T^{-n}(A)$ inductively on $n$ with $T^{-(n+1)}(A):=T^{-1}(T^{-n}(A))$.
For each $n\in \Z\setminus\{0\}$ and each $x\in X$, write $T^n(x):=T^{n}(\{x\})$. For a subset $Y\subset X$ and $x\in X$, define $T|_Y(x):=T(x)\cap Y$.

For each $n\in \N$, equip the product space $X^n:=\{(x_1,\ldots,x_n):x_i \in X, i=1,\ldots,n\}$ with the metric $d_n$ given by
\[
d_n((x_1,\dots,x_n),(y_1,\dots,y_n))=\max_{1\leq i\leq n}d(x_i,y_i)
\]
for all $(x_1,\dots,x_n), (y_1,\dots,y_n)\in X^n$.
Similarly, equip the product space
$X^\omega=\{(x_1,x_2,\dots): x_i \in X$ for all $i\in \N \}$ with the metric
$d_\infty$ given by
\[
d_\omega((x_1,x_2,\dots),(y_1,y_2,\dots))
=\sum_{i=1}^{\infty}\frac{d(x_i,y_i)}{2^i(1+d(x_i,y_i))}
\]
for all $(x_1,x_2,\dots),(y_1,y_2,\dots)\in X^\omega$. For each $n\in\hat{\N}$, the topology of $X^n$ induced by the metric $d_n$
is the product topology.

For each $n\in \N$, write
\begin{equation*}
\mathcal{O}_n(T)=\{(x_1,\dots,x_n)\in X^n: x_{i+1}\in T(x_i)
\text{ for each }i=1,\ldots,n-1\}.
\end{equation*}
The {\it orbit space} $\mathcal{O}_\omega(T)$ induced by $T$ is given by
\begin{equation*}
\mathcal{O}_\omega(T)= \{(x_1, x_2, \dots)\in X^\omega: x_{i+1}\in T(x_i)
\text{ for each } i\in \N\}.
\end{equation*}
For each $n\in\hat{\N}$, we call an element in $\mathcal{O}_n(T)$ an {\it orbit}. A sequence of orbits
$\underline{x}^{(n)}=(x_1^{(n)}, x_2^{(n)}, \ldots)$ in $X^\omega$
converges to an orbit $\underline{x}=(x_1, x_2, \ldots)\in X^\omega$
if and only if
$x_i^{(n)}$ converges to $x_i$ as $n\to +\infty$ for each $i\in \N$.

Let $\phi:\mathcal{O}_2(T)\to \R$ be a continuous function. Define
$\tilde{\phi}:\mathcal{O}_\omega(T)\to \R$ as follows:
\begin{equation}\label{eq:tilde of phi}
\tilde{\phi}(x_1,x_2,\ldots):=\phi(x_1,x_2).
\end{equation}
So $\tilde{\phi}$ is a continuous function on $\mathcal{O}_\omega(T)$.

Denote by $\tilde{\pi}_1, \tilde{\pi}_2: \bigcup_{n\in\hat{\N}\setminus\{1\}}X^n\to X$, and $\tilde{\pi}_{12}:\bigcup_{n\in\hat{\N}\setminus\{1\}}X^n\to X^2$ the projection maps given by
\[
\tilde{\pi}_{1}(x_n)_n=x_1, ~\tilde{\pi}_{2}(x_n)_n=x_2, ~\tilde{\pi}_{12}(x_n)_n=(x_1,x_2),
\]
respectively. Let $X$ be a compact metric space. If $\mu$ is a Borel probability measure on $X^n$ for some $n\in\hat{\N}\setminus\{1\}$, then $\mu\circ\tilde{\pi}_{12}^{-1}$ refers to a Borel probability measure on $X^2$ given by
$\mu\circ\tilde{\pi}_{12}^{-1}(A):=\mu(\tilde{\pi}_{12}^{-1}(A))$ for all $A\in \mathscr{B}(X^2)$, and $\mu\circ\tilde{\pi}_{i}^{-1}$ refers to a Borel probability measure on $X$ given by
$\mu\circ\tilde{\pi}_{i}^{-1}(A):=\mu(\tilde{\pi}_{i}^{-1}(A))$ for all $A\in \mathscr{B}(X)$ where $i=1,2$.

\subsection{Correspondences}

In this subsection, we state the definition of correspondences on compact metric spaces.

\begin{defn}
Let \((X,d)\) be a compact metric space. A map \( T: X \to \mathcal{F}(X) \) is called a \textit{correspondence} on \(X\) if for any \(x \in X\) and any open neighborhood \(U\) of \(T(x)\), there exists an open neighborhood \(V\) of \(x\) such that \(T(y) \subset U\) for all \(y \in V\).
\end{defn}

\begin{rem}
We provide several remarks.
\begin{enumerate}
  \item [(i)] By \cite[Theorems 1, 2, 3]{IM06}, a map \( T: X \to \mathcal{F}(X) \) is a correspondence if and only if the following two equivalent conditions hold:
      \begin{enumerate}
        \item The graph \(\mathcal{O}_2(T) = \{(x_1, x_2) \in X^2 : x_2 \in T(x_1)\}\) is closed in \(X^2\).
        \item \(\mathcal{O}_n(T)\) is closed in \(X^n\) for all \(n \in \mathbb{\widehat{N}}\).
      \end{enumerate}
  \item [(ii)] Let $T$ be a correspondence on a compact metric space $(X, d)$. If \(Y\subset X\) is a closed subset, then \( T|_Y: Y \to \mathcal{F}(Y) \) is a correspondence on \(Y\), where
      \(T|_Y(x)=T(x)\cap Y\) for each \(x\in Y\).
  \item [(iii)] Let $T$ be a correspondence on a compact metric space $(X, d)$. Recall
  \[
  T^{-1}(x)=\{y\in X:~x\in T(y)\} \text{ for all } x\in X.
  \]
  It follows from \cite[Lemma 4.4]{LLZ23} that if $T$ be a correspondence on $X$ satisfying $T(X)=X$, then so is $T^{-1}$.
\end{enumerate}
\end{rem}

Next we review the the concept of topological conjugacy between correspondences (see \cite{KT17}).

\begin{defn}
Let $T$ be a correspondence on a compact metric space $X$, and $S$ be a correspondence on a compact metric space $Y$. The correspondences $T$ and $S$ are said to be {\it topological conjugate} if there exists a homeomorphism $\theta: X\to Y$ such that $S\circ\theta= \theta\circ T$. In this case, $\theta$ is called a topological conjugacy between $T$ and $S$.
\end{defn}

Let $\theta: X\to Y$ be a map. For each $n\in\N$, define the map $\theta^{(n)}: X^n\to Y^n$ as
\[
\theta^{(n)}(x_1,\ldots,x_n)=(\theta(x_1),\ldots,\theta(x_n))
\text{ for all }(x_1,\ldots,x_n)\in\mathcal{O}_n(T).
\]
Moreover, define $\theta^{(\omega)}: \mathcal{O}_\omega(T)\to \mathcal{O}_\omega(S)$
as
\[
\theta^{(\omega)}(x_1,x_2,\ldots):=(\theta(x_1),\theta(x_2),\ldots)
\text{ for all }(x_1,x_2,\ldots)\in\mathcal{O}_\omega(T).
\]
It is not difficult to verify that if $\theta$ is continuous, then $\theta^{(n)}$ is continuous for each $n\in\hat{\N}$.

\subsection{Topological pressure for correspondences}

In this subsection, we recall the definition of topological pressure for correspondences introduced in \cite{LLZ23}.

Given a compact metric space \((X,d)\) and \(\epsilon>0\), we say that \(E\subset X\) is \(\epsilon\)-separated if for each pair of distinct
points \(x,y\in E\), we have \(d(x,y)\geq \epsilon\). We say that \(F\subset X\) is \(\epsilon\)-spanning if for each \(x\in X\) there exists \(y\in F\) such that \(d(x,y)< \epsilon\). For each continuous function \(\varphi: X\to \R\) and each \(\delta>0\), set \(\Delta(\varphi,\delta):=\sup\{|\varphi(x)-\varphi(y)|: x,y\in X \text { and } d(x,y)<\delta\}\), and $\|\varphi\|_\infty:=\sup\{|\varphi(x)|: x\in X\}$.

Let $T$ be a correspondence on a compact metric space $(X, d)$ and $\phi: \mathcal{O}_2(T) \rightarrow \mathbb{R}$ be a continuous function. For each
\(n\in\N\), the function \(S_n\phi: \mathcal{O}_{n+1}(T) \rightarrow \mathbb{R}\) is given by
\[
S_n\phi(x_1,\ldots,x_{n+1}) := \sum_{i=1}^{n} \phi(x_i, x_{i+1}).
\]
For each $n \in \mathbb{N}$ and $\epsilon > 0$, define
\[
s_n(T,\phi,\epsilon):= \sup\left\{ \sum_{x \in E} e^{S_n\phi(x)} : E \text{ is an $\epsilon$-separated subset of } \mathcal{O}_{n+1}(T) \right\},
\]
and
\[
r_n(T,\phi,\epsilon):= \inf\left\{ \sum_{x \in F} e^{S_n\phi(x)} : F \text{ is an $\epsilon$-spanning subset of } \mathcal{O}_{n+1}(T) \right\}.
\]

\begin{defn}
Let $ T $ be a correspondence on a compact metric space $ (X, d) $ and $ \phi: \mathcal{O}_2(T) \to \mathbb{R} $ be a continuous function. The \textit{ topological pressure} $ \Ptop(T, \phi) $ is defined as
\begin{align*}
\Ptop(T, \phi) &:= \lim_{\epsilon \to 0^+} \limsup_{n \to +\infty} \frac{1}{n} \log \left( \sup_{E_n(\epsilon)} \sum_{\underline{x} \in E_n(\epsilon)} e^{S_n\phi(x)} \right) \notag \\
&= \lim_{\epsilon \to 0^+} \limsup_{n \to +\infty} \frac{1}{n} \log \left( \inf_{F_n(\epsilon)} \sum_{\underline{x} \in F_n(\epsilon)} e^{S_n\phi(x)} \right),
\end{align*}
where $ E_n(\epsilon) $ ranges over all $\epsilon$-separated subsets of $ (\mathcal{O}_{n+1}(T), d_{n+1}) $ and $ F_n(\epsilon) $ ranges over all $\varepsilon$-spanning subsets of $ (\mathcal{O}_{n+1}(T), d_{n+1}) $.

In particular, if $ \phi \equiv 0 $, we call $ \Ptop(T, 0) $ the \textit{topological entropy} of $ T $ and denote it by $ \htop(T) $.
\end{defn}

\begin{rem}\
\begin{enumerate}
  \item[(i)] The above definition of topological pressure for correspondences is well-defined (see \cite[Definition 4.6]{LLZ23}).
  \item[(ii)] By \cite[Remark 4.7]{LLZ23} we know that $-\infty<P(T, \phi)\leq+\infty$.
\end{enumerate}
\end{rem}

If $T$ is a correspondence on a compact metric space $(X, d)$, then $(X^\omega,\sigma)$ is a topological dynamical system, where
$X^\omega$ is equipped with the metric $d_\omega$ and
$\sigma: X^\omega\to X^\omega$ is the shift map given by
\[
\sigma(x_1,x_2,x_3,\ldots):(x_2,x_3,\ldots)\text{ for all }(x_1,x_2,x_3,\ldots)\in X^\omega.
\]
Since $\mathcal{O}_\omega(T)$ is closed in $X^\omega$, $(\mathcal{O}_\omega(T),\sigma)$ is a topological dynamical subsystem.
Now we recall a characterization of the topological pressure for correspondences via the topological pressure of the system $(\mathcal{O}_\omega(T),\sigma)$.

\begin{lem}\label{lem:connection of two pressure}
Let $T$ be a correspondence on a compact metric space $(X, d)$ and $\phi: \mathcal{O}_2(T) \rightarrow \mathbb{R}$ be a continuous function. Then
\[
\Ptop(T,\phi)=\Ptop(\sigma,\tilde{\phi}),
\]
where $\Ptop(\sigma,\tilde{\phi})$ refers to the classical topological pressure of the dynamical system $(\mathcal{O}_\omega(T),\sigma)$ with the potential function $\tilde{\phi}$ given in \eqref{eq:tilde of phi} (see \cite{Wa82} for the definition of the classical topological pressure).
\end{lem}

The following basic properties of topological pressure for correspondences
are standard, which can be obtained from the well-known properties of topological pressure for single-valued continuous maps (see \cite[Theorem 9.7]{Wa82}) alongside Lemma \ref{lem:connection of two pressure}.

\begin{lem}\label{lem:basic properties of topological pressure for correspondences}
Let $T$ be a correspondence on a compact metric space $(X, d)$.
For any continuous functions $\phi, \varphi: \mathcal{O}_2(T) \rightarrow \mathbb{R}$, we have
\begin{enumerate}
  \item[(i)]
  $\Ptop(T,\phi)+\inf\varphi\leq\Ptop(T,\phi+\varphi)
  \leq\Ptop(T,\phi)+\sup\varphi$.
  In particular,
  \[
  \phi\leq\varphi \Rightarrow \Ptop(T,\phi)\leq \Ptop(T,\varphi).
  \]
  \item[(ii)] $\Ptop(T,\phi+c)=\Ptop(T,\phi)+c$, $\forall c\in \R$.
  \item[(iii)] $\Ptop(T,t\phi+(1-t)\varphi)\leq t \Ptop(T,\phi)+(1-t)\Ptop(T,\varphi)$, $\forall t\in [0,1]$.
  \item[(iv)] $\Ptop(T,\phi)=\Ptop(T,\phi+\psi\circ\tilde{\pi}_1-\psi\circ\tilde{\pi}_2)$, $\forall \psi\in C(X)$.
\end{enumerate}
\end{lem}

We now investigate the behavior of topological pressure for correspondences under topological conjugacy.
\begin{thm}\label{thm:topological pressure under topological conjugacy}
Let $T$ be a correspondence on a compact metric space $X$, $S$ be a correspondence on a compact metric space $Y$, and $\phi: \mathcal{O}_2(S) \rightarrow \mathbb{R}$ be a continuous function. If $T$ and $S$ are topologically conjugate via a homeomorphism $\theta: X\to Y$, then
\[
\Ptop(T,\varphi)=\Ptop(S,\phi),
\]
where $\varphi:=\phi\circ\theta^{(2)}|_{\mathcal{O}_2(T)}$. Especially,
\[
\htop(T)=\htop(S).
\]
\end{thm}

\begin{proof}
Let $\theta^{(\omega)}: \mathcal{O}_\omega(T)\to \mathcal{O}_\omega(S)$
be defined as
\[
\theta^{(\omega)}(x_1,x_2,\ldots):=(\theta(x_1),\theta(x_2),\ldots)
\text{ for all }(x_1,x_2,\ldots)\in\mathcal{O}_\omega(T).
\]
We verify that the map $\theta^{(\omega)}$ is well-defined and that
the following diagram commute.
$$
\begin{CD}
\mathcal{O}_\omega(T) @ > \theta^{(\omega)} >> \mathcal{O}_\omega(S) \\
@V \sigma|_T VV  @ VV \sigma|_S V @. \\
\mathcal{O}_\omega(T) @ > \theta^{(\omega)} >> \mathcal{O}_\omega(S)
\end{CD}
$$

For any sequence $(x_1,x_2,\ldots)\in \mathcal{O}_\omega(T)$, write $y_i=\theta(x_i)$ for $i\in\N$.
Then
\[
y_{i+1}=\theta(x_{i+1})\in \theta(T(x_i))=S(\theta(x_i))=S(y_i).
\]
This implies that $\theta^{(\omega)}(x_1,x_2,\ldots)=(y_1,y_2,\ldots)\in \mathcal{O}_\omega(S)$ and thus $\theta^{(\omega)}$ is well-defined and continuous.
Similarly, we can deduce that $\varphi$ is also well-defined. Moreover,
\[
\sigma|_S\circ\theta^{(\omega)}(x_1,x_2,\ldots)=(y_2,y_3,\ldots)
=\theta^{(\omega)}\circ\sigma|_T(x_1,x_2,\ldots).
\]
So the diagram commute. As $\theta^{(\omega)}$ is a homeomorphism,
the topological dynamical systems $(\mathcal{O}_\omega(T),\sigma|_T)$ and $(\mathcal{O}_\omega(S),\sigma|_S)$ are conjugate.

Note that $\tilde{\varphi}=\tilde{\phi}\circ\theta^{(\omega)}$. By Lemma \ref{lem:connection of two pressure} and \cite[Theorem 9.8]{Wa82} we have
\begin{align*}
\Ptop(T,\varphi)=\Ptop(\sigma|_T,\tilde{\varphi})
=\Ptop(\sigma|_S,\tilde{\phi})
=\Ptop(S,\phi).
\end{align*}
\end{proof}

\subsection{Transition probability kernels}

In this subsection, we recall the definition of transition probability kernels (see \cite{MT12} for more details), which are also called Markovian transition kernels (see \cite{Ga16}).

\begin{defn}
Let $(X, \mathscr{M}(X))$ and $(Y, \mathscr{M}(Y))$ be measurable spaces, where $X$ and $Y$ are sets and $\mathscr{M}(X)$ and $\mathscr{M}(Y)$ are \(\sigma\)-algebras on $X$ and $Y$, respectively. A {\it transition probability kernel from $Y$ to $X$} is a map $\mathcal{Q}: Y \times \mathscr{M}(X) \to [0,1]$ satisfying the following two properties:
\begin{enumerate}
    \item[(i)] For every \(y \in Y\), the map \(\mathscr{M}(X) \ni A \mapsto \mathcal{Q}(y,A)\) is a probability measure on \((X,\mathscr{M}(X))\).
    \item[(ii)] For every \(A \in \mathscr{M}(X)\), the map \(Y \ni y \mapsto \mathcal{Q}(y,A)\) is \(\mathscr{M}(Y)\)-measurable.
\end{enumerate}
For every $y \in Y$, denote by $\mathcal{Q}_y$ the probability measure on \((X,\mathscr{M}(X))\) such that $\mathcal{Q}_y(A) := \mathcal{Q}(y,A)$. If $Y = X$, then we call $\mathcal{Q}$ a {\it transition probability kernel} on $(X,\mathscr{M}(X))$, or simply on $X$ when the context is clear..
\end{defn}

\begin{defn}
Let \(T\) be a correspondence on a compact metric space \(X\), and let \(\mathcal{Q}\) be a transition probability kernel on \((X, \mathscr{B}(X))\), where \(\mathscr{B}(X)\) is the Borel \(\sigma\)-algebra
on \(X\). We say that \(\mathcal{Q}\) is \textit{supported by} \(T\) if \(\mathcal{Q}_x(T(x)) = 1\) for every \(x \in X\).
\end{defn}

Transition probability kernels generalize measurable maps and transition matrices. Their actions on functions and measures are standard. Now we recall them below.

\begin{defn}
Let \((X,\mathscr{M}(X))\) and \((Y,\mathscr{M}(Y))\) be measurable spaces.
\begin{enumerate}
  \item[(i)] Let \(f \) be a bounded measurable function on $X$, and \(\mathcal{Q}\) be a transition probability kernel from \(Y\) to \(X\). The \textit{pullback} function \(\mathcal{Q}f: Y \to \mathbb{R}\) of \(f\) by \(\mathcal{Q}\) is defined as:
      \[
      \mathcal{Q}f(y) := \int_{X} f(x) \, d\mathcal{Q}_{y}(x).
      \]
  \item[(ii)] Let \(\mu \) be a probability measure on $Y$, and \(\mathcal{Q}\) be a transition probability kernel from \(Y\) to \(X\). The \textit{pushforward} probability measure \(\mu\mathcal{Q}\) on \(X\) is defined as:
      \[
      (\mu\mathcal{Q})(A) := \int_{Y} \mathcal{Q}(y,A) \, d\mu(y) \quad \text{for all } A \in \mathscr{M}(X).
      \]
\end{enumerate}
\end{defn}

\begin{defn}
Let \((X,\mathscr{M}(X))\) be a measurable space and \(\mathcal{Q}\) be a transition probability kernel on \(X\). We say that a probability measure
$\mu$ on $X$ is \textit{\(\mathcal{Q}\)-invariant} if $\mu\mathcal{Q}=\mu$.
In particular, if $X$ is a compact metric space and \(\mathcal{Q}\) is a transition probability kernel on \((X,\mathscr{B}(X))\), we denote by \(\mathcal{P}_\mathcal{Q}(X)\) the set of all \(\mathcal{Q}\)-invariant Borel probability measures on \(X\).
\end{defn}

For each $n\in \N$, denote by $\mathscr{M} (X^n)$ the $\sigma$-algebra on $X^n$ generated by $\bigcup_{i=0}^{n-1} \{ X^i \times A \times X^{n-1-i}  :  A\in \mathscr{M} (X) \}$. Denote by $\mathscr{M} (X^\omega)$ the $\sigma$-algebra on $X^\omega$ generated by $\bigcup_{i=0}^{+\infty} \{ X^i \times A \times X^\omega  :  A\in \mathscr{M} (X) \}$.
For each $A_{n+1}\subset \mathscr{M}(X^{n+1})$ and each $(x_1, \ldots, x_n)\in X^n$, write
\begin{equation*}
\pi_{n+1} (x_1, \ldots, x_n;  A_{n+1}):= \{x_{n+1}\in X : (x_1, \ldots, x_n ,  x_{n+1})\in A_{n+1}\}.
\end{equation*}

Next we recall the definition of transition probability kernel $\mathcal{Q}^{[n]} (n\in\N_0)$ and $\mathcal{Q}^{[\omega]}$ induced by $\mathcal{Q}$.

\begin{defn}\label{Q^[n]}
Let $\mathcal{Q}$ be a transition probability kernel on a measurable space $(X, \mathscr{M}(X))$, where $X$ is a set and $\mathscr{M}(X)$ is a $\sigma$-algebra on $X$. Define the transition probability kernel $\mathcal{Q}^{[n]}$ from $X$ to $X^{n+1}$ inductively on $n\in \N_0$ as follows:

	First, let $\mathcal{Q}^{[0]}_x =\delta_x$, the Dirac measure at $x\in X$, for all $x\in X$. If $\mathcal{Q}^{[n-1]}$ has been defined for some $n\in \N$, then we define $\mathcal{Q}^{[n]}$ as:
\begin{equation*}
\mathcal{Q}^{[n]} (x,  A_{n+1}):= \int_{X^n} \mathcal{Q} (x_n,  \pi_{n+1} (x_1, \ldots, x_n;  A_{n+1}))  \,d\mathcal{Q}_x^{[n-1]}
(x_1, \ldots, x_n)
\end{equation*}
for all $x\in X$ and $A_{n+1}\in \mathscr{M} \bigl(X^{n+1} \bigr)$.
\end{defn}

\begin{defn}\label{Q^omega}
Let $\mathcal{Q}$ be a transition probability kernel on a measurable space $(X, \mathscr{M}(X))$. Define the transition probability kernel $\mathcal{Q}^\omega$ from $X$ to $X^\omega$ as the unique transition probability kernel from $X$ to $X^\omega$ with the property that for each $x\in X$, each $n\in \N_0$, and each measurable set $A\in \mathscr{M}(X^{n+1})$, the following equality holds:
\begin{equation*}\label{Q^N(x,A*X^infty)=Q^[n](x,A)}
\mathcal{Q}^\omega (x, A\times X^\omega)= \mathcal{Q}^{[n]}(x, A).
\end{equation*}
\end{defn}

The following lemma is adapted from \cite[Lemmas 6.3 and A.9]{LLZ23} and will be used extensively throughout this paper.
\begin{lem}\label{lem:a integral equality}
Let \(T\) be a correspondence on a compact metric space \(X\), $\mathcal{Q}$ a transition probability kernel on $X$ supported by $T$, $\mu\in \mathcal{P}(X)$, and $\phi\in C(\mathcal{O}_2(T))$. Then
\[
\int_{\mathcal{O}_2(T)}\phi\,d(\mu\mathcal{Q}^{[1]})
=\int_X\int_{T(x_1)}\phi(x_1,x_2)\,d\mathcal{Q}_{x_1}(x_2)\,d\mu(x_1).
\]
\end{lem}

\subsection{Measure-theoretic entropy for transition probability kernels}
In this subsection, we recall the definition of measure-theoretic entropy for transition probability kernels introduced in \cite{LLZ23}, which has been proven to generalize the measure-theoretic entropy of measurable maps.

A {\it finite measurable partition} $\mathcal{A}$ of a measurable space $(X,  \mathscr{M}(X))$ is a finite collection of mutually disjoint measurable subsets $\{A_1 ,\, \dots ,\, A_n\}$ satisfying $\bigcup_{i=1}^n A_i =X$, where $n \in \N$. For a finite measurable partition $\mathcal{A}$, let
\begin{equation*}
\mathcal{A}^n:=\underbrace{\mathcal{A} \times \cdots \times \mathcal{A}}_n
	:= \{A_1 \times \cdots \times A_n  :  A_i \in \mathcal{A}  \text{ for every } 1\leq i\leq n  \} \subset \mathscr{M} (X^{n}).
\end{equation*}
It is clear that $\mathcal{A}^n$ is a finite measurable partition of $(X^{n}, \mathscr{M} (X^{n}))$.

\begin{defn}\label{h_mu (Q,A)}
Let $\mathcal{Q}$ be a transition probability kernel on a measurable space $(X, \mathscr{M}(X))$, $\mu$ be a $\mathcal{Q}$-invariant probability measure on $X$, and $\mathcal{A}$ be a finite measurable partition of $X$.
\begin{enumerate}
  \item[(i)] The {\it measure-theoretic entropy $h_\mu (\mathcal{Q},\mathcal{A})$ of $\mathcal{Q}$ w.r.t. $\mathcal{A}$}, is defined as
\begin{equation*}
h_\mu (\mathcal{Q},\mathcal{A}):= \lim_{n\to +\infty} \frac{1}{n}
H_{\mu \mathcal{Q}^{[n-1]}} (\mathcal{A}^n),
\end{equation*}
where $H_{\mu} (\mathcal{A})$ is defined as
\[
H_\mu (\mathcal{A}):= -\sum_{A\in \mathcal{A}}   \mu (A) \log (\mu (A)).
\]
  \item[(ii)] The {\it measure-theoretic entropy $h_\mu (\mathcal{Q})$ of $\mathcal{Q}$ for $\mu$} is defined as
      \[
      h_\mu (\mathcal{Q}):= \sup_{\mathcal{A}} h_\mu (\mathcal{Q},\mathcal{A}),
      \]
      where $\mathcal{A}$ ranges over all finite measurable partitions of $X$.
\end{enumerate}

\end{defn}

\section{Variational principle (I)}

\subsection{Measure-theoretic entropy of invariant measures for correspondences}

In this subsection, we introduce the measure-theoretic entropy of invariant measures for correspondences. First, we recall some foundational results concerning the invariant measures for correspondences.

Motivated by \cite{MA99} and \cite{LLZ23}, we derive the following equivalent characterizations, which are very important for the subsequent discussion. For the reader's convenience, a self-contained proof is provided.

\begin{lem}\label{lem:invariant-measure-characterization-correspondence}
Let \(T\) be a correspondence on a compact metric space \(X\). For a measure \(\mu \in \mathcal{P}(X)\) the following conditions are equivalent:

\begin{enumerate}
    \item[(i)] For every Borel set \( A \subset X \), it holds that
    \[
        \mu(A) \leq \mu(T^{-1}(A)).
    \]

    \item[(ii)] There exists a transition probability kernel \(\mathcal{Q}\) on \(X\) supported by $T$ such that
    \[
    \mu = \mu\mathcal{Q}.
    \]

    \item[(iii)] There exists a measure \(\tilde{\mu} \in \mathcal{P}(X^2)\) such that \(\tilde{\mu}(\mathcal{O}_2(T))=1\) and
        \[
        \mu =\tilde{\mu}\circ\tilde{\pi}_1^{-1}
        =\tilde{\mu}\circ\tilde{\pi}_2^{-1}.
        \]

    \item[(iv)] There exists a \( \sigma \)-invariant measure \(\nu \in \mathcal{P}_\sigma(X^\omega)\) which is supported on \( \mathcal{O}_\omega(T) \) and satisfies:
        \[
        \mu=\nu\circ\tilde{\pi}_1^{-1}.
        \]
\end{enumerate}
Moreover, the set of all \(\mu \in \mathcal{P}(X)\) that satisfy one of the above equivalent conditions is compact and convex in \( \mathcal{P}(X)\).
\end{lem}
\begin{proof}
(iv) $\Rightarrow$ (iii): Let $\tilde{\mu}=v\circ\tilde{\pi}_{12}^{-1}$.
It can be seen that \(\tilde{\mu}(\mathcal{O}_2(T))=1\). In addition,
we have
\[
\tilde{\mu}\circ\tilde{\pi}_1^{-1}
=(v\circ\tilde{\pi}_{12}^{-1})\circ\tilde{\pi}_1^{-1}
=v\circ\tilde{\pi}_1^{-1}=\mu,
\]
and
\[
\tilde{\mu}\circ\tilde{\pi}_2^{-1}
=(v\circ\tilde{\pi}_{12}^{-1})\circ\tilde{\pi}_2^{-1}
=v\circ\tilde{\pi}_2^{-1}=(v\circ\sigma^{-1})\circ\tilde{\pi}_1^{-1}
=v\circ\tilde{\pi}_1^{-1}=\mu.
\]

(iii) $\Rightarrow$ (ii):
Applying \cite[Proposition A.11]{LLZ23} with $M=\mathcal{O}_2(T)$, we can find a transition probability kernel \(\mathcal{Q}\) on \(X\) such that $\mathcal{Q}$ is supported by $T$, $\mu=\tilde{\mu}\circ\tilde{\pi}_1^{-1}$ and $\tilde{\mu}=\mu\mathcal{Q}^{[1]}$. Furthermore, by \cite[Corollary A.7]{LLZ23} we have $\mu=\tilde{\mu}\circ\tilde{\pi}_2^{-1}
=(\mu\mathcal{Q}^{[1]})\circ \tilde{\pi}_2^{-1}=\mu\mathcal{Q}$.

(ii) $\Rightarrow$ (iv): Let $v=\mu\mathcal{Q}^{\omega}$. Then $\mu\mathcal{Q}^{\omega}$ is \( \sigma \)-invariant (see \cite[Section 5.4]{LLZ23}) and $\mu=\nu\circ\tilde{\pi}_1^{-1}$. It follows from \cite[Lemma 6.13]{LLZ23}
that $\mu\mathcal{Q}^{\omega}$ is supported on \( \mathcal{O}_\omega(T) \).

Finally, the equivalence of conditions (i) and (iii) follows immediately from \cite[Theorem 3.2]{MA99}.
\end{proof}

Now we recall the definition of invariant measures for correspondences, which comes from \cite{MA99}.

\begin{defn}\label{defn:invariant-measure-correspondence}
    Let $T$ be a correspondence on a compact metric space $X$. A Borel probability measure $\mu$ on $X$ is called \textit{$ T $-invariant} if
    $$
    \mu(A) \leq \mu(T^{-1}(A)) \quad \text{for all Borel sets } A \subset X.
    $$
    We denote by $\mathcal{P}_T(X)$ the set of all $T$-invariant Borel probability measures on $X$. For $\mu \in \mathcal{P}_T(X)$, let $\mathcal{K}_\mu$ denote the set of all transition probability kernels $\mathcal{Q}$ on $X$ such that:
    \begin{enumerate}
        \item[(i)] $\mathcal{Q}$ is supported by $T$ (i.e., $\mathcal{Q}_x(T(x)) = 1$ for every $x \in X$), and
        \item[(ii)] $\mu$ is $\mathcal{Q}$-invariant (i.e., $\mu = \mu \mathcal{Q}$).
    \end{enumerate}
\end{defn}

\begin{rem}\label{rem:remark of invariant measure for correspondences}
Some remarks are in order.
    \begin{enumerate}
        \item[(i)] Lemma \ref{lem:invariant-measure-characterization-correspondence} provides several equivalent characterizations of $T$-invariant measures and
            establishes that $\mathcal{P}_T(X)$ is compact and convex within $\mathcal{P}(X)$, the space of Borel probability measures on $X$ equipped with the weak* topology.

        \item[(ii)] Let $\mathcal{P}^e_T(X)$ denote the set of {\it extreme points} of the compact convex set $\mathcal{P}_T(X)$. By the Choquet representation theorem, for every $\mu \in \mathcal{P}_T(X)$, there exists a probability measure $\mathbb{P}_\mu$ on the Borel $\sigma$-algebra of $\mathcal{P}_T(X)$ such that $\mathbb{P}_\mu(\mathcal{P}^e_T(X)) = 1$ and
            \begin{equation*}\label{eq:extremal-decomposition}
                \int_X \psi  d\mu = \int_{\mathcal{P}^e_T(X)} \left( \int_X \psi  dm \right) d\mathbb{P}_\mu(m)
            \end{equation*}
            holds for every continuous function $\psi \in C(X)$. We express this as
            $$\mu = \int_{\mathcal{P}^e_T(X)} m  d\mathbb{P}_\mu(m)$$ and call it the \textit{extremal decomposition of $\mu$}.
    \end{enumerate}
\end{rem}

\begin{defn}\label{defn:measure-theoretic entropy of invariant measures for correspondences}
Let $T$ be a correspondence on a compact metric space $(X, d)$ and $\phi: \mathcal{O}_2(T) \rightarrow \mathbb{R}$ be a continuous function. For $\mu\in\mathcal{P}_T(X)$, we define the {\it measure-theoretic pressure of $\phi$ for $\mu$} and the {\it measure-theoretic entropy of $\mu$} as follows:
\[
P_\mu (T,\phi)=\sup_{\mathcal{Q}\in \mathcal{K}_\mu}\left\{h_\mu (\mathcal{Q})+ \int_X \int_{T(x_1)} \phi (x_1, x_2) \, d\mathcal{Q}_{x_1}(x_2) \, d\mu (x_1)\right\},
\]
and
\[
h_\mu (T)=\sup_{\mathcal{Q}\in \mathcal{K}_\mu}\{h_\mu (\mathcal{Q})\}.
\]
\end{defn}

Based on the above definition, we can restate Theorem \ref{thm:variational-principle-correspondence} as follows:

\begin{thm}\label{thm:restatement of variational principle for correspondence}
Let $(X, d)$ be a compact metric space, $T$ be a forward expansive correspondence on $X$, and $\phi: \mathcal{O}_2(T) \rightarrow \mathbb{R}$ be a continuous function.
\begin{enumerate}
  \item[(i)] The variational principle for topological pressure holds:
\[
\Ptop(T, \phi) = \sup_{\mu \in \mathcal{P}_T(X)} \left\{ P_\mu (T,\phi) \right\}.
\]
Moreover, this supremum can be attained at some $\mu \in \mathcal{P}_T(X)$.
  \item[(ii)] The variational principle for topological entropy holds:
\[
\htop(T) = \sup_{\mu \in \mathcal{P}_T(X)} \left\{ h_\mu (T) \right\}.
\]
Moreover, this supremum can be attained at some $\mu \in \mathcal{P}_T(X)$.
\end{enumerate}
\end{thm}

We now examine the behavior of the measure-theoretic pressure (entropy) under topological conjugacy.

\begin{thm}\label{thm:measure pressure under topological conjugacy}
Let $T$ be a correspondence on a compact metric space $X$, let $S$ be a correspondence on a compact metric space $Y$, let $\mu$ be a $T$-invariant measure, and let $\phi: \mathcal{O}_2(S) \rightarrow \mathbb{R}$ be a continuous function. If $T$ and $S$ are topologically conjugate via a homeomorphism $\theta: X\to Y$, then $\mu\circ\theta^{-1}$ is an $S$-invariant measure and
\[
P_\mu(T,\varphi)=P_{\mu\circ\theta^{-1}}(S,\phi),
\]
where $\varphi:=\phi\circ\theta^{(2)}|_{\mathcal{O}_2(T)}$. Especially,
\[
h_\mu(T)=h_{\mu\circ\theta^{-1}}(S).
\]
\end{thm}
\begin{proof}
By the $T$-invariance of $\mu$, Lemma \ref{lem:invariant-measure-characterization-correspondence} implies the existence a transition probability kernel $\mathcal{Q}$ on $X$ supported by $T$ satisfying $\mu\mathcal{Q}=\mu$.
Let $\mathcal{L}: Y\times \mathscr{B}(Y)\to [0,1]$ be defined as follows:
for any $y\in Y$ and $B\in \mathscr{B}(Y)$, set
\[
\mathcal{L}(y,B):=\mathcal{Q}(\theta^{-1}(y),\theta^{-1}(B)).
\]
It is not difficult to see that $\mathcal{L}$ is a transition probability kernel on $Y$ supported by $S$. We divide the remaining proof into several steps.

\emph{Step 1. The measure $(\mu\circ\theta^{-1})$ is $S$-invariant.}

Since
\begin{align*}
((\mu\circ\theta^{-1})\mathcal{L})(B)
&=\int_Y\mathcal{L}(y,B)\,d\mu\circ\theta^{-1}(y)\\
&=\int_X\mathcal{L}(\theta(x),B)\,d\mu(x)\\
&=\int_X\mathcal{Q}(x,\theta^{-1}(B))\,d\mu(x)\\
&=(\mu\mathcal{Q})(\theta^{-1}(B))\\
&=(\mu\circ\theta^{-1})(B),
\end{align*}
we deduce that $(\mu\circ\theta^{-1})$ is $\mathcal{L}$-invariant and hence $S$-invariant
by Lemma \ref{lem:invariant-measure-characterization-correspondence}.

\emph{Step 2. By Theorem \ref{thm:topological pressure under topological conjugacy}, the function $\varphi$ is well-defined and continuous.}

\emph{Step 3.
For each $n\in\N$, the following equality
\begin{equation}\label{eq:relation of Q and L}
\mathcal{Q}^{[n]}(x,A_{n+1})=\mathcal{L}^{[n]}(\theta x,\theta^{(n+1)}(A_{n+1}))
\end{equation}
holds for all $A_{n+1}\in\mathscr{B}(X^{n+1})$. In other word,
\[
\mathcal{Q}^{[n]}_x\circ(\theta^{(n+1)})^{-1}=\mathcal{L}^{[n]}_{\theta x}.
\]
Furthermore, one has
\[
(\mu\mathcal{Q}^{[n]})(A_{n+1})
=((\mu\circ\theta^{-1})\mathcal{L}^{[n]})(\theta^{(n+1)}(A_{n+1})).
\]}

For $n=1$ and $A_{2}\in\mathscr{B}(X^{2})$, we have
\begin{align*}
\mathcal{Q}^{[1]}(x,A_{2})&=\int_X\mathcal{Q}(x_1,\pi_2(x_1;A_{2}))
\,d\mathcal{Q}^{[0]}_x(x_1)\\
&=\mathcal{Q}(x,\pi_2(x;A_{2}))\\
&=\mathcal{L}(\theta x,\theta(\pi_2(x;A_{2})))\\
&=\mathcal{L}(\theta x,\pi_2(\theta x;\theta^{(2)}(A_{2})))\\
&=\int_Y\mathcal{L}(x_1,\pi_2(x_1;\theta^{(2)}(A_{2})))
\,d\mathcal{L}^{[0]}_{\theta x}(x_1)\\
&=\mathcal{L}^{[1]}(\theta x,\theta^{(2)}(A_{2})).
\end{align*}
So \eqref{eq:relation of Q and L} holds for $n=1$.

We assume that \eqref{eq:relation of Q and L} holds for $n=k$,
and prove that it also holds for $n=k+1$.
Next, denote $\underline{x}_1^n:=(x_1,\ldots,x_{n})\in X^n$ and $\underline{y}_1^n:=(y_1,\ldots,y_{n})\in Y^n$ for $n\in\N$.
For $A_{k+2}\in\mathscr{B}(X^{k+2})$, we have
\begin{align*}
\mathcal{Q}^{[k+1]}(x,A_{k+2})&=\int_{X^{k+1}}
\mathcal{Q}(x_{k+1},\pi_{k+2}(\underline{x}_1^{k+1};A_{k+2}))
\,d\mathcal{Q}^{[k]}_x(\underline{x}_1^{k+1})\\
&=\int_{X^{k+1}}
\mathcal{L}(\theta(x_{k+1}),\pi_{k+2}
(\theta^{(k+1)}(\underline{x}_1^{k+1});\theta^{(k+2)}(A_{k+2})))
\,d\mathcal{Q}^{[k]}_x(\underline{x}_1^{k+1})\\
&=\int_{Y^{k+1}}
\mathcal{L}(y_{k+1},\pi_{k+2}
(\underline{y}_1^{k+1};\theta^{(k+2)}(A_{k+2})))
\,d\mathcal{Q}^{[k]}_x\circ(\theta^{(k+1)})^{-1}(\underline{y}_1^{k+1})\\
&=\int_{Y^{k+1}}
\mathcal{L}(y_{k+1},\pi_{k+2}
(\underline{y}_1^{k+1};\theta^{(k+2)}(A_{k+2})))
\,d\mathcal{L}^{[k]}_{\theta x}(\underline{y}_1^{k+1})\\
&=\mathcal{L}^{[k+1]}(\theta x,\theta^{(k+2)}(A_{k+2})).
\end{align*}
Hence \eqref{eq:relation of Q and L} holds for $n=k+1$.
Furthermore, we get
\begin{align*}
(\mu\mathcal{Q}^{[n]})(A_{n+1})&=\int_X\mathcal{Q}^{[n]}(x,A_{n+1})\,d\mu(x)\\
&=\int_X\mathcal{L}^{[n]}(\theta x,\theta^{(n+1)}(A_{n+1}))\,d\mu(x)\\
&=\int_Y\mathcal{L}^{[n]}(y,\theta^{(n+1)}(A_{n+1}))\,d(\mu\circ\theta^{-1})(y)\\
&=((\mu\circ\theta^{-1})\mathcal{L}^{[n]})(\theta^{(n+1)}(A_{n+1})).
\end{align*}
This ends the proof of step 3.

\emph{Step 4. We prove that $P_\mu(T,\varphi)=P_{\mu\circ\theta^{-1}}(S,\phi)$.}

Let $\mathcal{A}=\{A_1,\ldots,A_k\}$ be a finite Borel measurable partition of $X$. Define
\[
\mathcal{B}:=\{\theta(A_1),\ldots,\theta(A_k)\},
\]
which forms a finite measurable partition of $Y$. From step 3 we have
\[
H_{\mu\mathcal{Q}^{[n-1]}}(\mathcal{A}^n)
=H_{(\mu\circ\theta^{-1})\mathcal{L}^{[n-1]}}(\mathcal{B}^n).
\]
It follows immediately that $h_\mu(\mathcal{Q},\mathcal{A})
=h_{\mu\circ\theta^{-1}}(\mathcal{L},\mathcal{B})$. Consequently,
\[
h_\mu(\mathcal{Q})
\leq h_{\mu\circ\theta^{-1}}(\mathcal{L}).
\]
Moreover, by Lemma \ref{lem:a integral equality} and step 3,
\begin{align*}
\int_X \int_{T(x_1)} \varphi (x_1, x_2) \, d\mathcal{Q}_{x_1}(x_2) \, d\mu (x_1)
&=\int_{\mathcal{O}_2(T)}\phi\circ\theta^{(2)}\,d(\mu\mathcal{Q}^{[1]})\\
&=\int_{\mathcal{O}_2(S)}\phi\,d(\mu\mathcal{Q}^{[1]}\circ(\theta^{(2)})^{-1})\\
&=\int_{\mathcal{O}_2(S)}\phi\,d((\mu\circ\theta^{-1})\mathcal{L}^{[1]})\\
&=\int_Y \int_{S(y_1)} \phi (y_1, y_2) \, d\mathcal{L}_{y_1}(y_2)
\, d(\mu\circ\theta^{-1}) (x_1).
\end{align*}
Therefore
\[
P_\mu(T,\varphi)\leq P_{\mu\circ\theta^{-1}}(S,\phi).
\]
By symmetry of the conjugacy $\theta$, it holds that
$P_\mu(T,\varphi)= P_{\mu\circ\theta^{-1}}(S,\phi)$.
\end{proof}

\subsection{Variational principle (I)}

In this subsection, we provide a partial solution to Question \ref{ques:1}. Specifically, we introduce a class of correspondences and prove that for such systems, the topological pressure $\Ptop(T, \phi)$ is determined by the measure-theoretic pressure over $\mathcal{P}^e_T(X)$,  the extreme points of the space of $T$-invariant Borel probability measures.

The following definition draws inspiration from the Lee--Lyubich--Markorov--Mazor--Mukherjee anti-holomorphic correspondences in complex dynamics.

\begin{defn}
Let $(X, d)$ be a compact metric space and $T$ be a correspondence on $X$.
We say that \(T\) is {\it generated by $(X_1, T_1)\to (X_2, T_2)\to\cdots \to(X_d, T_d)$}
if the following several conditions hold.
\begin{enumerate}
  \item[(i)] \(X=\bigcup_{i=1}^d X_i\).

  \item[(ii)] \(X_i\) is a closed subset of \(X\) for every $i=1,\ldots,d$.

  \item[(iii)] $T_i$ is a correspondence on $X_i$ for every $i=1,\ldots,d$.

  \item[(iv)] $T|_{X_i}=T_i$ for every $i=1,\ldots,d$. In other word, \(T(x)\cap X_i =T_i(x)\) for \(x\in X_i\).
  \item[(v)] $T(X_i)\cap \left(\bigcup_{j= 1}^{i-1} X_i\setminus X_i\right)=\emptyset$ for every $i=2,\ldots,d$.
\end{enumerate}
\end{defn}

\begin{rem}
We give some comments for the above definition.
\begin{enumerate}
  \item[(i)] If the correspondence \(T\) on a compact metric space $X$ is generated by $(X_1, T_1)\to (X_2, T_2)\to\cdots \to(X_d, T_d)$, then we call $(X_1, T_1)\to (X_2, T_2)\to\cdots \to(X_d, T_d)$ a decomposition of the correspondence \(T\) on $X$.

  \item[(ii)] Let $T$ be a correspondence on $X$ generated by $(X_1, T_1)\to (X_2, T_2)\to\cdots \to(X_d, T_d)$ and $\phi: \mathcal{O}_2(T) \rightarrow \mathbb{R}$ be a continuous function. Define
      $\phi_{T_i}: \mathcal{O}_2(T_i) \rightarrow \mathbb{R}$ as follows:
      \[
      \phi_{T_i}(x_1,x_2):=\phi(x_1,x_2) \text{ for all }(x_1,x_2)\in\mathcal{O}_2(T_i)\subset \mathcal{O}_2(T).
      \]
      Clearly, $\phi_{T_i}$ is also a continuous function for every $i=1,\ldots,d$.
\end{enumerate}
\end{rem}

The following lemma is standard, and we omit its proof here for brevity.
\begin{lem}\label{lem:limit lemma}
Let $\{a_n\}_{n\geq 1}$ and $\{b_n\}_{n\geq 1}$ be sequences such that
$a_n>0$ and $b_n>0$ for all $n$. Then
\[
\limsup_{n\to\infty}\frac{1}{n}\log\sum_{k=1}^na_kb_{n-k}
=\max\left\{\limsup_{n\to\infty}\frac{1}{n}\log a_n, \limsup_{n\to\infty}\frac{1}{n}\log b_n\right\}.
\]
\end{lem}

\begin{lem}\label{lem:Entropy formula for two subsystems}
Let $(X, d)$ be a compact metric space, $T$ be a correspondence on $X$ generated by $(X_1, T_1)\to (X_2, T_2)$, and $\phi: \mathcal{O}_2(T) \rightarrow \mathbb{R}$ be a continuous function. Then
\[
\Ptop(T,\phi)=\max\{\Ptop(T_1,\phi_{T_1}),\Ptop(T_2,\phi_{T_2})\}.
\]
\end{lem}
\begin{proof}
Recall that
\begin{equation*}
\Ptop(T_1,\phi_{T_1}) = \lim_{\epsilon \to 0} \limsup_{n \to \infty} \frac{1}{n} \log \alpha(n,\epsilon),
\end{equation*}
where
\begin{equation*}
\alpha(n,\epsilon) = \sup_{E_n(\epsilon)} \sum_{\underline{x} \in E_n(\epsilon)} e^{S_n\phi_{T_1}(\underline{x})}=\sup_{E_n(\epsilon)} \sum_{\underline{x} \in E_n(\epsilon)} e^{S_n\phi(\underline{x})}
\end{equation*}
and $E_n(\epsilon)$ ranges over all $\epsilon$-separated subsets of $(\mathcal{O}_{n+1}(T_1), d_{n+1})$.
Meanwhile,
\begin{equation*}
\Ptop(T_2,\phi_{T_2}) = \lim_{\epsilon \to 0} \limsup_{n \to \infty} \frac{1}{n} \log \beta(n,\epsilon),
\end{equation*}
where
\begin{equation*}
\beta(n,\epsilon) = \sup_{F_n(\epsilon)} \sum_{\underline{x} \in F_n(\epsilon)} e^{S_n\phi_{T_2}(\underline{x})} =\sup_{F_n(\epsilon)} \sum_{\underline{x} \in F_n(\epsilon)} e^{S_n\phi(\underline{x})}
\end{equation*}
and $F_n(\epsilon)$ ranges over all $\epsilon$-separated subsets of $(\mathcal{O}_{n+1}(T_2), d_{n+1})$. Besides, we let $\alpha(0,\epsilon)$ and $\beta(0,\epsilon)$ denote the maximum cardinality among all $\epsilon$-separated subsets of $X$.

For any $n\in\N$ and $\epsilon$-separated set $W_n(\epsilon) \subset \mathcal{O}_{n+1}(T)$, we define for any $k\in\{-1,0,\ldots,n\}$ that
$$
W_{n,k}(\epsilon) := \left\{ (y_0,\ldots,y_n) \in W_n(\epsilon) :
y_i \in X\setminus X_2 \text{ for }i \leq k,\
y_i \in X_2 \text{ for }i > k
\right\}.
$$
Then it is obvious that
$$
W_n(\epsilon) = \bigcup_{k=-1}^n W_{n,k}(\varepsilon).
$$
Let $E_k(\epsilon/2)$ be a $\epsilon/2$-separated subset of $(\mathcal{O}_{k+1}(T_1), d_{k+1})$ with maximal cardinality.
For $(x_0,\ldots,x_k) \in E_k(\epsilon/2)$, we define
$$
W_{n,k,x_0,\ldots,x_k}(\epsilon) := \left\{ (y_0,\ldots,y_n) \in W_{n,k}(\epsilon) : d(x_i,y_i) < \epsilon/2 \text{ for all } 0\leq i \leq k \right\}.
$$
The maximality of $E_k(\epsilon/2)$ implies that
\begin{equation}\label{eq:definition of W}
W_{n,k}(\epsilon)=\bigcup_{(x_0,\ldots,x_k)\in E_k(\epsilon/2)} W_{n,k,x_0,\ldots,x_k}(\epsilon).
\end{equation}

Fix $0\leq k \leq n - 1$ and $(x_0,\cdots,x_k) \in E_k(\epsilon/2)$. For any $\underline{y} = (y_0,\cdots,y_n) \in W_{n,k,x_0,\ldots,x_k}(\epsilon)$, it can be shown that
$$
\sum_{j=0}^{k-1} \phi(y_j,y_{j+1})
\leq \sum_{j=0}^{k-1} \phi(x_j,x_{j+1})
+ k\Delta(\phi,\epsilon/2),
$$
where
$$
\Delta(\phi,\delta) = \sup\left\{ |\phi(x_1,x_2) - \phi(y_1,y_2)| : d(x_1,y_1) < \delta,\ d(x_2,y_2) < \delta \right\}\ \text{for all } \delta > 0.
$$
Therefore, we have
\begin{equation*}
S_n\phi(\underline{y}) \leq \sum_{j=0}^{k-1} \phi(x_j,x_{j+1}) + k\Delta\left(\phi,\frac{\epsilon}{2}\right) + \|\phi\|_\infty
+ \sum_{j=k+1}^{n-1}\phi(y_j,y_{j+1}).
\end{equation*}
Since $W_{n,k,x_0,\ldots,x_k}(\epsilon)$ is an $\epsilon$-separated subset of $\mathcal{O}_{n+1}(T)$, for any two distinct orbits $(y_0,\ldots,y_n)$, $(z_0,\ldots,z_n) \in W_{n,k,x_0,\ldots,x_k}(\epsilon)$, there exists $l \in \{0,\ldots,n\}$ such that $d(y_l,z_l) \geq \epsilon$. It is not difficult to verify that this $l$ must belong to $\{k+1,\ldots,n\}$.
Therefore, the projection set
$$
\left\{ (y_{k+1},\ldots,y_n) : (y_0,\ldots,y_n) \in W_{n,k,x_0,\ldots,x_k}(\epsilon) \right\}
$$
forms an $\epsilon$-separated subset of $\mathcal{O}_{n-k}(T_2)$.
Thus,
$$
\sum_{\underline{y}\in W_{n,k,x_0,\ldots,x_k}(\epsilon)} e^{S_n\phi(\underline{y})}
\leq \beta(n-k-1,\epsilon)\exp\left( \sum_{j=0}^{k-1} \phi(x_j,x_{j+1}) + k\Delta\left(\phi,\frac{\epsilon}{2}\right) + \|\phi\|_\infty \right),
$$
which together with \eqref{eq:definition of W} implies that
\begin{equation*}
\sum_{\underline{y}\in W_{n,k}(\epsilon)} e^{S_n\phi(\underline{y})}\leq \alpha\left(k, \frac{\epsilon}{2}\right) \beta(n - k - 1, \epsilon) \exp\left(n\Delta\left(\phi, \frac{\epsilon}{2}\right) + \|\phi\|_\infty \right).
\end{equation*}

Next we consider $k = -1$ and $k = n$ independently. Note that
$$
W_{n,-1}(\epsilon) = \left\{ (y_0,\ldots,y_n) \in W_n(\epsilon) : y_i \in X_2 \ \text{for all}\ i =0,\ldots,n \right\}
$$
is $\epsilon$-separated in $\mathcal{O}_{n+1}(T_2)$, and
$$
W_{n,n}(\epsilon) = \left\{ (y_0,\ldots,y_n) \in W_n(\epsilon) : y_i \in X_1 \ \text{for all}\ i =0,\ldots,n \right\}
$$
is $\epsilon$-separated in $\mathcal{O}_{n+1}(T_1)$. Hence, we have
$$
\sum_{\underline{y} \in W_{n,-1}(\epsilon)} \exp(S_n\phi(\underline{y})) \leq \beta(n, \epsilon)
$$
and
$$
\sum_{\underline{y} \in W_{n,n}(\epsilon)} \exp(S_n\phi(\underline{y})) \leq \alpha(n, \epsilon).
$$

Furthermore, from $W_n(\epsilon) = \bigcup_{k=-1}^n W_{n,k}(\epsilon)$, we conclude that
\[
\sum_{\underline{y} \in W_n(\epsilon)} \exp(S_n\phi(\underline{y})) \leq \alpha(n,\epsilon) + \beta(n,\epsilon) + e^{n\Delta(\phi,\frac{\epsilon}{2})+\|\phi\|_\infty} \sum_{k=0}^{n-1} \alpha\left(k, \frac{\epsilon}{2}\right)\beta(n-k-1,\epsilon).
\]
According to Lemma \ref{lem:limit lemma}, it follows immediately that
\begin{align*}
P(T, \phi) &= \lim_{\epsilon \to 0} \limsup_{n \to \infty} \frac{1}{n} \log \left(\sup_{W_n(\epsilon)}\sum_{\underline{y} \in W_n(\epsilon)} \exp(S_n\phi(\underline{y}))\right) \\
&\leq \max\left\{ \Ptop(T_1,\phi_{T_1}), \Ptop(T_2,\phi_{T_2}) \right\}.
\end{align*}

On the other hand, it is straightforward to observe that for any $n \in \mathbb{N}$ and $\epsilon > 0$, each $\epsilon$-separated subset of $\mathcal{O}_{n+1}(T_1)$ is also an $\epsilon$-separated subset of $\mathcal{O}_{n+1}(T)$. Thus,
$$
\Ptop(T, \phi) \geq \Ptop(T_1, \phi|_{\mathcal{O}_2(T_1)})
=\Ptop(T_1, \phi|_{T_1}).
$$
Similarly, one has
$$
\Ptop(T, \phi) \geq \Ptop(T_2, \phi|_{T_2}).
$$
So
$$
P(T, \phi)\geq \max\left\{ \Ptop(T_1,\phi_{T_1}), \Ptop(T_2,\phi_{T_2}) \right\}.
$$
\end{proof}

From the preceding lemma, we deduce the following theorem.

\begin{thm}[Pressure formula for decompositions]
\label{thm:Entropy formula for subsystems}
Let $(X, d)$ be a compact metric space, $T$ be a correspondence on $X$ generated by $(X_1, T_1)\to (X_2, T_2)\to\cdots \to(X_d, T_d), d\geq 2$, and $\phi: \mathcal{O}_2(T) \rightarrow \mathbb{R}$ be a continuous function. Then
\[
\Ptop(T,\phi)=\max_{1\leq i\leq d}\{\Ptop(T_i,\phi_{T_i})\}.
\]
\end{thm}
\begin{proof}
When $d=2$, the result follows from Lemma \ref{lem:Entropy formula for two subsystems}. Assume that the conclusion holds for $d=k$, we shall prove that it is also true for $d=k+1$.

Denote $Y=X_1\cup X_2$ and define $S: Y\to \mathcal{F}(Y)$ by $S(x)=T(x)\cap Y$ for all $x\in Y$. Clearly, $S$ is a correspondence on $Y$ and
$T$ is generated by $(Y, S)\to (X_3, T_3)\to\cdots \to(X_d, T_d)$.
By the inductive hypothesis, we have
\[
\Ptop(T,\phi)=\max\left\{\Ptop(S,\phi_{S}), \max_{2\leq i\leq d}\left\{\Ptop(T_i,\phi_{T_i})\right\}\right\}.
\]
It remains to prove that
\[
\Ptop(S,\phi_{S})=\max\{\Ptop(T_1,\phi_{T_1}),\Ptop(T_2,\phi_{T_2})\}.
\]

For $i=1,2$ and $x\in X_i$, we get $S(x)\cap X_i=T(x)\cap Y \cap X_i=T_i(x)\cap Y=T_i(x)$, which gives $S|_{X_i}=T_i$.
Besides, for $x\in X_2$ we have
$$
S(x)=T(x)\cap Y\subset (X_2\cup X_1^c)\cap Y=X_2.
$$
So $S$ is generated by $(X_1, T_1)\to (X_2, T_2)$.
Denote by $\varphi=\phi_S$. Then for any
\[
(x_1,x_2)\in \mathcal{O}_2(T_1)
\subset \mathcal{O}_2(S)\subset \mathcal{O}_2(T),
\]
it holds that
\[
\varphi_{T_1}(x_1,x_2)=\varphi(x_1,x_2)=\phi_S(x_1,x_2)
=\phi(x_1,x_2)=\phi_{T_1}(x_1,x_2).
\]
This leads to $\varphi_{T_1}=\phi_{T_1}$.
Similarly, one can get $\varphi_{T_2}=\phi_{T_2}$.
Applying Lemma \ref{lem:Entropy formula for two subsystems}, we derive
\[
\Ptop(S,\phi_{S})=\Ptop(S,\varphi)
=\max\{\Ptop(T_1,\phi_{T_1}),\Ptop(T_2,\phi_{T_2})\}.
\]
This ends the proof.
\end{proof}

The following two lemmas are essential to this section's main result.
\begin{lem}\label{lem:key-extreme-relation1}
Let $T$ be a correspondence on a compact metric space $(X, d)$, $f$ be a continuous self-map on a closed subset $Y\subset X$ and $\phi: \mathcal{O}_2(T) \rightarrow \mathbb{R}$ be a continuous function.
If $T|_Y=\C_f$ and $\mu\in\mathcal{P}_f^e(Y)$, then we have $\hat{\mu}\in\mathcal{P}_T^e(X)$ and
\[
P_{\hat{\mu}}(T,\phi)=P_\mu(f,\phi_f):=h_\mu(f)+\int_Y\phi_f\,d\mu,
\]
where $\phi_f$ is the continuous function on $Y$ defined by
$\phi_f(x):=\phi(x,f(x))$ for all $x\in Y$
and
$\hat{\mu}$ is the probability measure on $X$ defined by
\[
\hat{\mu}(A):=\mu(A\cap Y) \text{ for any }A\in \mathscr{B}(X).
\]
\end{lem}
\begin{proof}
The proof proceeds in following two steps.

\emph{Step 1. Given $\mu\in\mathcal{P}_f(Y)$, prove that $\hat{\mu}$ is a $T$-invariant Borel probability measure on $X$ satisfying $P_{\hat{\mu}}(T,\phi)=h_\mu(f)+\int_Y\phi_f\,d\mu$.}

Let $\mathcal{Q}_f$ be the transition probability kernel on $Y$ induced by $f$ (see \cite[Definition B.1]{LLZ23}). By \cite[Lemma B.2 and (B.5)]{LLZ23},
we conclude that $\mu$ is $\mathcal{Q}_f$-invariant and $h_{\mu}(f) = h_{\mu}(\mathcal{Q}_f)$. Moreover, by \cite[Lemma 1.1]{MA99}, one can choose a Borel measurable selection map $t_1: X\to X$ for $T$. Next, for any $x\in A$ and any $A\in\mathscr{B}(X)$, define
$$
\mathcal{Q}(x, A) = \begin{cases}
1_A(f(x)) & \text{ if } x \in Y, \\
1_A(t_1(x)) & \text{ if } x \in X\setminus Y.
\end{cases}
$$
We observe that $\mathcal{Q}$ is a transition probability kernel on $X$
supported by $T$. Besides, it follows from \cite[Lemma 5.28]{LLZ23} that
$\hat{\mu}$ is $\mathcal{Q}$-invariant.
According to Lemma \ref{lem:invariant-measure-characterization-correspondence}, we deduce that $\hat{\mu}$ is $T$-invariant.

Let $\mathcal{S}$ be a transition probability kernel on $X$
supported by $T$ satisfying $\hat{\mu}\mathcal{S}=\hat{\mu}$. Then from
\begin{align*}
1=\hat{\mu}(Y)=(\hat{\mu}\mathcal{S})(Y)
&=\int_X\mathcal{S}(x,Y)\,d\hat{\mu}(x)\\
&=\int_Y\mathcal{S}(x,Y)\,d\mu(x)
\end{align*}
we conclude that the equality $\mathcal{S}(x,Y)=1$ holds for $\mu$-almost
every $x\in Y$. Consequently,
\[
\mathcal{S}_x(f(x))=\mathcal{S}_x(T(x)\cap Y)=1\text{ for } \mu\text{-almost
every }x\in Y.
\]
It follows that
$$\mathcal{S}_x(A)=\delta_{f(x)}(A)=\mathcal{Q}_f(x,A)$$
holds for all $A\in \mathscr{B}(Y)$ and $\mu$-almost every $x\in Y$. Then \cite[Lemma 5.28]{LLZ23} yields
$$
h_{\hat{\mu}}(\mathcal{S}) = h_{\mu}(\mathcal{Q}_f)= h_{\mu}(f).
$$
Furthermore, we derive
\begin{align*}
\int_X \int_{T(x_1)} \phi (x_1, x_2) \, d\mathcal{S}_{x_1}(x_2) \,
d\hat{\mu} (x_1)
&=\int_Y \int_{T(x_1)} \phi (x_1, x_2) \, d\delta_{f(x_1)}(x_2) \,d\hat{\mu} (x_1)\\
&=\int_Y \phi (x_1, f(x_1)) \,d \mu (x_1)\\
&=\int_Y \phi_f \,d\mu.
\end{align*}
Finally, since the choice of $\mathcal{S}$ was arbitrary, we obtain
\[
P_{\hat{\mu}}(T,\phi)=\sup_{\mathcal{S}\in \mathcal{K}_{\hat{\mu}}}\left\{h_{\hat{\mu}} (\mathcal{S})+ \int_X \int_{T(x_1)} \phi (x_1, x_2) \, d\mathcal{S}_{x_1}(x_2) \, d\hat{\mu} (x_1)\right\}= h_\mu(f)+\int_Y\phi_f\,d\mu.
\]

\emph{Step 2. Given $\mu\in\mathcal{P}_f^e(Y)$, prove that $\hat{\mu}\in\mathcal{P}_T^e(X)$.}

We use proof by contradiction to prove this assertion. Denote by $\nu=\hat{\mu}$ and assume that there exist
$p\in(0,1)$ and $\nu_1, \nu_2\in\mathcal{P}_T(X)$ such that
$\nu=p\nu_1+(1-p)\nu_2$. From $v(Y)=1$ it is not difficult to verify that $\nu_1(Y)=\nu_2(Y)=1$. Define $\nu_1|_Y, \nu_2|_Y\in\mathcal{P}(Y)$ as follows:
\[
\nu_1|_Y(A):=\frac{\nu_1(A)}{\nu_1(Y)}=\nu_1(A),~
\nu_2|_Y(A):=\frac{\nu_2(A)}{\nu_2(Y)}=\nu_2(A) \text{ for } A\in \mathscr{B}(Y).
\]
By the property of the correspondence $T$, it follows directly that for every $A\in \mathscr{B}(Y)$,
\begin{align*}
T^{-1}(A)\cap Y&=\{x\in X:~T(x)\cap A\neq\emptyset\}\cap Y\\
&=\{x\in Y:~T(x)\cap A\cap Y\neq\emptyset\}\\
&=\{x\in Y:~f(x)\in A\}\\
&=f^{-1}(A).
\end{align*}
Since $\nu_1$ is $T$-invariant we obtain
\[
\nu_1(A)\leq\nu_1(T^{-1}(A))=\nu_1(T^{-1}(A)\cap Y)=\nu_1(f^{-1}(A))
\text{ for any }A\in \mathscr{B}(Y),
\]
which means that
\[
\nu_1|_Y(A)\leq\nu_1|_Y(f^{-1}(A))
\]
holds for any $A\in \mathscr{B}(Y)$.

Furthermore, for $A\in \mathscr{B}(Y)$ let
\[
\omega(A):=\nu_1|_Y(f^{-1}(A))-\nu_1|_Y(A)\geq0.
\]
For disjoint Borel subsets $A$ and $B$ in $Y$, we can verify the following equation:
\begin{align*}
\omega(A\cup B)&=\nu_1|_Y(f^{-1}(A\cup B))-\nu_1|_Y(A\cup B)\\
&=\nu_1|_Y(f^{-1}(A))+\nu_1|_Y(f^{-1}(B))-\nu_1|_Y(A)-\nu_1|_Y (B)\\
&=\omega(A)+\omega(B).
\end{align*}
Therefore, for any $A\in \mathscr{B}(Y)$ we can get
\[
\omega(A)\leq\omega(Y)=\nu_1|_Y(f^{-1}(Y))-\nu_1|_Y(Y)=0,
\]
which yields that
\[
\nu_1|_Y(A)=\nu_1|_Y(f^{-1}(A)).
\]
So $\nu_1|_Y$ is $f$-invariant. Similarly, it can be shown that $\nu_2|_Y$ is also $f$-invariant.

Now, it follows from $\nu=p\nu_1+(1-p)\nu_2$ that
$\mu=\nu|_Y=p\nu_1|_Y+(1-p)\nu_2|_Y$. Since $\nu_1|_Y, \nu_2|_Y\in\mathcal{P}_f(Y)$ it means that $\mu$ is not an extreme
point of $\mathcal{P}_f(Y)$, a contradiction.
\end{proof}

\begin{lem}\label{lem:key-extreme-relation2}
Let $T$ be a correspondence on a compact metric space $(X, d)$ satisfying $T(X)=X$, $g$ be a continuous self-map on a closed subset $Z\subset X$ and $\phi: \mathcal{O}_2(T) \rightarrow \mathbb{R}$ be a continuous function.
If $T|_Z=\C_g^{-1}$ and $\mu\in\mathcal{P}_g^e(Z)$, then we have $\hat{\mu}\in\mathcal{P}_T^e(X)$ and
\[
P_{\hat{\mu}}(T,\phi)=P_\mu(g,\phi_g):=h_\mu(g)+\int_Z\phi_g\,d\mu,
\]
where $\phi_g$ is the continuous function on $Z$ defined by
$\phi_g(x):=\phi(g(x),x)$ for all $x\in Z$
and $\hat{\mu}$ is the probability measure on $X$ defined by
\[
\hat{\mu}(A):=\mu(A\cap Z) \text{ for any }A\in \mathscr{B}(X).
\]
\end{lem}
\begin{proof}
We proceed to prove this conclusion through the following two steps.

\emph{Step 1. Given $\mu\in\mathcal{P}_g(Z)$, prove that $\hat{\mu}$ is a $T$-invariant Borel probability measure on $X$ satisfying $P_{\hat{\mu}}(T,\phi)=h_\mu(g)+\int_Z\phi_g\,d\mu$.}

Let $\mathcal{Q}_g$ be the transition probability kernel on $Y$ induced by $g$ (see \cite[Definition B.1]{LLZ23}). By \cite[Lemma B.2 and (B.5)]{LLZ23},
we conclude that $\mu$ is $\mathcal{Q}_g$-invariant and $h_{\mu}(g) = h_{\mu}(\mathcal{Q}_g)$. Moreover, by \cite[Lemma 1.1]{MA99}, one can choose a Borel measurable selection map $t_2: X\to X$ for $T^{-1}$. Next, for any $x\in A$ and any $A\in\mathscr{B}(X)$, define
$$
\mathcal{Q}(x, A) = \begin{cases}
1_A(g(x)) & \text{ if } x \in Z, \\
1_A(t_2(x)) & \text{ if } x \in X\setminus Z.
\end{cases}
$$
We observe that $\mathcal{Q}$ is a transition probability kernel on $X$
supported by $T^{-1}$. By \cite[Lemma 6.13]{LLZ23}, the measure $\hat{\mu}\mathcal{Q}^{[1]}$ is supported on $\mathcal{O}_2(T^{-1})$.
Moreover, \cite[Lemma 5.28]{LLZ23} establishes that
$\hat{\mu}$ is $\mathcal{Q}$-invariant.
From \cite[(A.10)]{LLZ23}, it follows that $(\hat{\mu}\mathcal{Q}^{[1]})\circ\tilde{\pi}_2^{-1}
=\hat{\mu}\mathcal{Q}=\hat{\mu}$. Furthermore, in light of \cite[Proposition A.11 and Definition A.14]{LLZ23}, one can choose a backward conditional transition probability kernel $\mathcal{R}$ of $\hat{\mu}\mathcal{Q}^{[1]}$ from
$X$ to $X$, supported by $\mathcal{O}_2(T^{-1})$.
The transition probability kernel $\mathcal{R}$ is supported by $T$ and satisfies
$(\hat{\mu}\mathcal{Q}^{[1]})\circ\gamma_2^{-1}
=\hat{\mu}\mathcal{R}^{[1]}$, which together with \cite[Proposition 5.23]{LLZ23} yields that
$\hat{\mu}$ is $\mathcal{R}$-invariant. Finally,
by Lemma \ref{lem:invariant-measure-characterization-correspondence}, we conclude that $\hat{\mu}$ is $T$-invariant.

Let $\mathcal{S}$ be a transition probability kernel on $X$
supported by $T$ satisfying $\hat{\mu}\mathcal{S}=\hat{\mu}$.
By \cite[Lemma 6.13]{LLZ23} and \cite[(A.10)]{LLZ23}, we deduce that
$\hat{\mu}\mathcal{S}^{[1]}$ is supported on $\mathcal{O}_2(T)$ and
$(\hat{\mu}\mathcal{S}^{[1]})\circ\tilde{\pi}_2^{-1}
=\hat{\mu}\mathcal{S}=\hat{\mu}$.
Then, invoking \cite[Proposition A.11 and Definition A.14]{LLZ23}, we may select a backward conditional transition probability kernel $\mathcal{L}$ of $\hat{\mu}\mathcal{S}^{[1]}$ from
$X$ to $X$ such that the kernel $\mathcal{L}$ is supported by $T^{-1}$ and satisfies the property that
$$(\hat{\mu}\mathcal{S}^{[1]})\circ\gamma_2^{-1}
=\hat{\mu}\mathcal{L}^{[1]}.$$
Combining this with \cite[Proposition 5.23]{LLZ23} implies that $\hat{\mu}$ is $\mathcal{L}$-invariant and
$$
h_{\hat{\mu}}(\mathcal{L}) = h_{\hat{\mu}}(\mathcal{S}).
$$
Moreover, by
\begin{align*}
1=\hat{\mu}(Z)=(\hat{\mu}\mathcal{L})(Z)
&=\int_X\mathcal{L}(x,Z)\,d\hat{\mu}(x)\\
&=\int_Z\mathcal{L}(x,Z)\,d\mu(x)
\end{align*}
we have $\mathcal{L}(x,Z)=1$ for $\mu$-almost
every $x\in Z$. Thus, for such $x\in Z$, we get
\[
\mathcal{L}_x(g(x))=\mathcal{L}_x(T^{-1}(x)\cap Z)=1.
\]
Consequently,
$$\mathcal{L}_x(A)=\delta_{g(x)}(A)=\mathcal{Q}_g(x,A)$$
holds for all $A\in \mathscr{B}(Z)$ and $\mu$-almost every $x\in Z$. Then applying \cite[Lemma 5.28]{LLZ23} we obtain
$$
h_{\hat{\mu}}(\mathcal{S})=h_{\hat{\mu}}(\mathcal{L}) = h_{\mu}(\mathcal{Q}_g)= h_{\mu}(g).
$$
Furthermore, by lemma \ref{lem:a integral equality}, it can be shown that
\begin{align*}
\int_X \int_{T(x_1)} \phi (x_1, x_2) \, d\mathcal{S}_{x_1}(x_2) \,
d\hat{\mu} (x_1)&=\int_{\mathcal{O}_2(T)}\phi\, d(\hat{\mu}\mathcal{S}^{[1]})\\
&=\int_{\mathcal{O}_2(T)}\phi\circ\gamma_2\, d(\hat{\mu}\mathcal{L}^{[1]})\\
&=\int_{\mathcal{O}_2(T)}\phi(x_2,x_1)\, d(\hat{\mu}\mathcal{L}^{[1]})(x_1,x_2)\\
&=\int_X \int_{T^{-1}(x_1)} \phi (x_2, x_1) \, d\mathcal{L}_{x_1}(x_2) \,
d\hat{\mu} (x_1)\\
&=\int_Z \int_{T^{-1}(x_1)} \phi (x_2, x_1) \, d\delta_{g(x_1)}(x_2) \,d\mu (x_1)\\
&=\int_Z \phi_g \,d\mu.
\end{align*}
Consequently,
\[
P_{\hat{\mu}}(T,\phi)=\sup_{\mathcal{S}\in \mathcal{K}_{\hat{\mu}}}\left\{h_{\hat{\mu}} (\mathcal{S})+ \int_X \int_{T(x_1)} \phi (x_1, x_2) \, d\mathcal{S}_{x_1}(x_2) \, d\hat{\mu} (x_1)\right\}= h_\mu(g)+\int_Z\phi_g\,d\mu.
\]

\emph{Step 2. Given $\mu\in\mathcal{P}_g^e(Z)$, prove that $\hat{\mu}\in\mathcal{P}_T^e(X)$.}

We use proof by contradiction to prove this assertion. Denote by $\nu=\hat{\mu}$ and assume that there exist
$p\in(0,1)$ and $\nu_1, \nu_2\in\mathcal{P}_T(X)$ such that
$\nu=p\nu_1+(1-p)\nu_2$. From $v(Z)=1$ it is not difficult to verify that $\nu_1(Z)=\nu_2(Z)=1$. Define $\nu_1|_Z, \nu_2|_Z\in\mathcal{P}(Z)$ as follows:
\[
\nu_1|_Z(A):=\frac{\nu_1(A)}{\nu_1(Z)}=\nu_1(A),~
\nu_2|_Z(A):=\frac{\nu_2(A)}{\nu_2(Z)}=\nu_2(A) \text{ for } A\in \mathscr{B}(Z).
\]
By the property of the correspondence $T$, it is obvious for any $A\in \mathscr{B}(Z)$ that
\begin{align*}
T^{-1}(A)\cap Z&=\{x\in X:~T(x)\cap A\neq\emptyset\}\cap Z\\
&=\{x\in Z:~T(x)\cap A\cap Z\neq\emptyset\}\\
&=\{x\in Z:~g^{-1}(x)\cap A\neq\emptyset\}\\
&=g(A).
\end{align*}
Since $\nu_1$ is $T$-invariant we obtain
\[
\nu_1(A)\leq\nu_1(T^{-1}(A))=\nu_1(T^{-1}(A)\cap Z)=\nu_1(g(A))
\text{ for any }A\in \mathscr{B}(Z),
\]
which means that
\[
\nu_1|_Z(g^{-1}(A))\leq\nu_1|_Z((A))
\]
holds for any $A\in \mathscr{B}(Z)$.

Furthermore, for $A\in \mathscr{B}(Z)$, let
\[
\omega(A):=\nu_1|_Z(A)-\nu_1|_Z(g^{-1}(A))\geq0.
\]
For disjoint Borel subsets $A$ and $B$ in $Z$, we can verify the following equation:
\begin{align*}
\omega(A\cup B)&=\nu_1|_Z(A\cup B)-\nu_1|_Z(g^{-1}(A\cup B))\\
&=\nu_1|_Z(A)+\nu_1|_Z (B)-\nu_1|_Z(g^{-1}(A))-\nu_1|_Z(g^{-1}(B))\\
&=\omega(A)+\omega(B).
\end{align*}
Therefore, for any $A\in \mathscr{B}(Z)$ we can get
\[
\omega(A)\leq\omega(Z)=\nu_1|_Z(Z)-\nu_1|_Z(g^{-1}(Z))=0,
\]
which yields that
\[
\nu_1|_Z(A)=\nu_1|_Z(g^{-1}(A)).
\]
So $\nu_1|_Z$ is $g$-invariant. Similarly, it can be shown that $\nu_2|_Z$ is also $g$-invariant.

Now, it follows from $\nu=p\nu_1+(1-p)\nu_2$ that
$\mu=\nu|_Z=p\nu_1|_Z+(1-p)\nu_2|_Z$. Since $\nu_1|_Z, \nu_2|_Z\in\mathcal{P}_g(Z)$ it means that $\mu$ is not an extreme
point of $\mathcal{P}_g(Z)$, a contradiction.
\end{proof}

Now we present the main result of this section.

\begin{thm}\label{thm:type I variation principle (over extreme points)}
Let $(X, d)$ be a compact metric space, $T$ be a correspondence on $X$ satisfying $T(X)=X$ and $\phi: \mathcal{O}_2(T) \rightarrow \mathbb{R}$ be a continuous function. Suppose that $T$ is generated by $(X_1, T_1)\to (X_2, T_2)\to\cdots \to(X_d, T_d), d\geq 1$,
where $T_i=\C_{f_i}$ or $\C_{f_i}^{-1}$, $f_i$ is a continuous self-map on $X_i$. Then
\[
\Ptop(T, \phi) = \sup_{\mu \in \mathcal{P}_T(X)} \left\{ P_\mu (T,\phi) \right\}=\sup_{\mu \in \mathcal{P}^e_T(X)} \left\{ P_\mu (T,\phi) \right\}.
\]
Especially,
\[
\htop(T) = \sup_{\mu \in \mathcal{P}_T(X)} \left\{ h_\mu (T) \right\}=\sup_{\mu \in \mathcal{P}^e_T(X)} \left\{ h_\mu (T) \right\}.
\]
\end{thm}
\begin{proof}
By Theorem \ref{thm:Entropy formula for subsystems}, we obtain
\[
\Ptop(T,\phi)=\max_{1\leq i\leq d}\{\Ptop(T_i,\phi_{T_i})\}.
\]
Then there exists $1\leq i\leq d$ such that $\Ptop(T,\phi)=\Ptop(T_i,\phi_{T_i})$.

When $T_i=\C_{f_i}$. Define \(\phi_{f_i}\in C(X_i)\) by
\(\phi_{f_i}(x)=\phi(x,f_i(x))=\phi_{T_i}(x,f_i(x))\) for all \(x\in X_i\). Then \cite[Proposition B.3]{LLZ23}
implies $\Ptop(T_i,\phi_{T_i})=\Ptop(f_i,\phi_{f_i})$.
By the classical variational principle, we obtain
\[
\Ptop(f_i,\phi_{f_i})=\sup_{\mu\in\mathcal{P}_{f_i}(X_i)}
\left\{h_\mu(f_i)+\int_{X_i}\phi_{f_i}\,d\mu\right\}
=\sup_{\mu\in\mathcal{P}_{f_i}^e(X_i)}
\left\{h_\mu(f_i)+\int_{X_i}\phi_{f_i}\,d\mu\right\}.
\]
For each $\mu\in\mathcal{P}_{f_i}^e(X_i)$, let $\hat{\mu}$ denote the probability measure on $X$ defined by
\[
\hat{\mu}(A):=\mu(A\cap X_i) \text{ for any }A\in \mathscr{B}(X).
\]
According to Lemma \ref{lem:key-extreme-relation1}, we have
$\hat{\mu}\in\mathcal{P}_T^e(X)$ and
\[
P_{\hat{\mu}}(T,\phi)=h_\mu(f_i)+\int_{X_i}\phi_{f_i}\,d\mu.
\]
So
\[
\Ptop(T,\phi)\leq\sup_{\mu\in\mathcal{P}_{T}^e(X)}P_{\mu}(T,\phi),
\]
which together with \cite[Theorem D]{LLZ23} yields the conclusion.

When $T_i=\C_{f_i}^{-1}$. Define \(\phi_{f_i}\in C(X_i)\) by
\(\phi_{f_i}(x)=\phi(f_i(x),x)=\phi_{T_i}(f_i(x),x)\) for all \(x\in X_i\). Then \cite[Propositions 4.8 and B.3]{LLZ23}
implies $\Ptop(T_i,\phi_{T_i})=\Ptop(f_i,\phi_{f_i})$.
By applying Lemma \ref{lem:key-extreme-relation2} and adopting a method analogous to the previous case, we can complete the remaining proof.
\end{proof}

As applications of Theorem \ref{thm:type I variation principle (over extreme points)}, we present several interesting corollaries and examples. The following corollary follows from \cite[Proposition 4.8]{LLZ23}, Lemma \ref{lem:key-extreme-relation2}, and Theorem \ref{thm:type I variation principle (over extreme points)}.

\begin{cor}
Let $X$ be a compact metric space and $f$ be a continuous surjective map on $X$. The following conclusions hold.
\begin{enumerate}
  \item[(i)] $\htop(f)=\htop(f^{-1})$.
  \item[(ii)] A Borel probability measure $\mu$ on $X$ is $f$-invariant if and only if it is $f^{-1}$-invariant. Moreover, for such measures, we have $h_\mu(f)=h_{\mu}(f^{-1})$.
  \item[(iii)] $\htop(f^{-1})=\sup_{\mu\in \mathcal{P}_{f^{-1}}(X)}\{h_{\mu}(f^{-1})\}=\sup_{\mu\in \mathcal{P}^e_{f^{-1}}(X)}\{h_{\mu}(f^{-1})\}$.
\end{enumerate}
\end{cor}

\begin{rem}
The results presented above extend the classical entropy theory of homeomorphisms to the setting of non-invertible maps.
\end{rem}

From \cite{LLZ23, LMM23}, one can see that the Lee--Lyubich--Markorov--Mazor--Mukherjee anti-holomorphic correspondence satisfies the conditions of Theorem \ref{thm:type I variation principle (over extreme points)}. Therefore, the following result can be directly established.

\begin{cor}
Let $\mathfrak{C}^*$ be the Lee--Lyubich--Markorov--Mazor--Mukherjee anti-holomorphic correspondence on $\mathbb{\widehat{C}}$ and $\phi: \mathcal{O}_2(\mathfrak{C}^*) \rightarrow \mathbb{R}$ be a continuous function. Then
\[
\Ptop(\mathfrak{C}^*,\phi)=\sup_{\mu\in P_{\mathfrak{C}^*}(\mathbb{\widehat{C}})}\{P_\mu(\mathfrak{C}^*,\phi)\}
=\sup_{\mu\in P^e_{\mathfrak{C}^*}(\mathbb{\widehat{C}})}\{P_\mu(\mathfrak{C}^*,\phi)\}.
\]
\end{cor}

The following is a simple example that satisfies the condition of Theorem \ref{thm:type I variation principle (over extreme points)}.

\begin{exmp}
Let $X=[0,1]$ and $f,g:X\to X$ be defined as
\begin{equation*}
  f(x) =
  \begin{cases}
    x, & 0\leq x \leq \frac{1}{2} \\
    \frac{1}{2}x+\frac{1}{4}, & \frac{1}{2}\leq x \leq 1,
  \end{cases}
\end{equation*}
and
\begin{equation*}
  g(x) =
  \begin{cases}
    -2x+1, & 0\leq x \leq \frac{1}{4} \\
    2x, & \frac{1}{4}\leq x \leq \frac{1}{2} \\
    -\frac{1}{2}x+\frac{5}{4}, & \frac{1}{2}\leq x \leq 1.
  \end{cases}
\end{equation*}
\end{exmp}

\begin{figure*}[h]
\centering 
\begin{tikzpicture}
\begin{axis}[
    axis equal image, 
    axis lines = middle,
    xlabel = $x$,
    ylabel = $y$,
    xmin = 0, xmax = 1,
    ymin = 0, ymax = 1,
    grid = both,
    title = {$f(x)$ and $g(x)$},
]

\addplot[blue, thick, domain=0:0.5] {x};
\addplot[blue, thick, domain=0.5:1] {0.5*x + 0.25};

\addplot[red, thick, domain=0:0.25] {-2*x + 1};
\addplot[red, thick, domain=0.25:0.5] {2*x};
\addplot[red, thick, domain=0.5:1] {-0.5*x + 1.25};

\addplot[black, dashed, thick] coordinates {(0.5,0) (0.5,1)}; 
\addplot[black, dashed, thick] coordinates {(0,0.5) (1,0.5)}; 

\node[anchor=south east] at (axis cs:0.5,1.4) {$x=\frac{1}{2}$};
\node[anchor=west] at (axis cs:1.1,0.5) {$y=\frac{1}{2}$};

\node[blue] at (axis cs:0.7,0.6) {$f(x)$};
\node[red] at (axis cs:0.75,0.875) {$g(x)$};
\end{axis}
\end{tikzpicture}
\end{figure*}

Let $T(x)=\{f(x),g(x)\}$ for $x\in X$. Then $T$ is a correspondence on $X$. We consider
\[
X_1:=\left[0,\frac{1}{2}\right],\text{ and }X_2:=\left[\frac{1}{2},1\right],
\]
and define the maps
$h_1:X_1\to X_1$ and $h_2:X_2\to X_2$ given by
$h_1(x)=\{x\}$, and
\begin{equation*}
  h_2(x) =
  \begin{cases}
    2x-\frac{1}{2}, & \frac{1}{2}\leq x \leq \frac{3}{4}, \\
    -2x+\frac{5}{2}, & \frac{3}{4}\leq x \leq 1. \\
  \end{cases}
\end{equation*}
Then
\[
h_2^{-1}(x)=\left\{\frac{1}{2}x+\frac{1}{4},-\frac{1}{2}x+\frac{5}{4}\right\} \text{ for }
x\in X_2.
\]
It is easy to verify that $T$ is a correspondence on $X$ generated by $(X_1, \mathcal{C}_{h_1})\to (X_2, \mathcal{C}_{h_2}^{-1})$. By Theorem \ref{thm:type I variation principle (over extreme points)}, we have
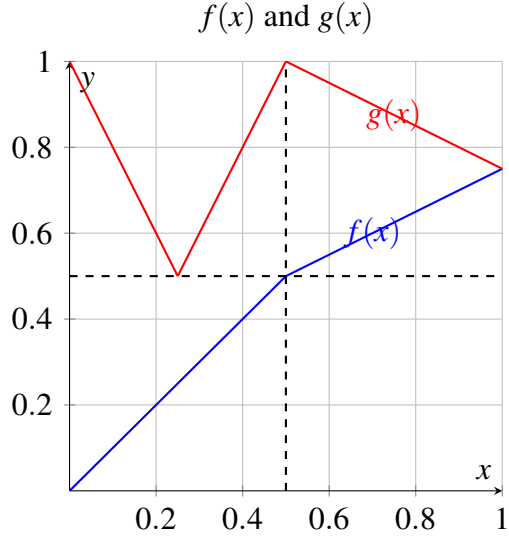
\[
\htop(T) = \sup_{\mu \in \mathcal{P}_T(X)} \left\{ h_\mu (T) \right\}=\sup_{\mu \in \mathcal{P}^e_T(X)} \left\{ h_\mu (T) \right\}
=\log 2.
\]

\section{Variational principle (II)}

The concept of (topological) invariance entropy in control systems originated from the seminal work of Nair et al. \cite{NEMM04} and was further developed by Colonius and Kawan \cite{CK08,Ka13}. As a natural generalization of invariance entropy, invariance pressure was introduced and has been extensively studied (see \cite{CCS18,CCS20,CCS22,CSC19,ZH19}).
Subsequently, various notions of measure-theoretic invariance entropy and a series of corresponding variational principles were established in the literature (see, for example, \cite{Co18-1,Co18-2,WHS19,WH22}). In 2022, building on a fundamental fact that topological pressure determines measure-theoretic entropy, Nie, Wang, and Huang \cite{NWH22} established a variational principle linking invariance pressure and measure-theoretic invariance entropy in control systems, utilizing tools from functional analysis.

Independently, Bi\'{s}, Carvalho, Mendes, and Varandas \cite{BCMV22} employed a similar approach to establish an abstract variational principle via convex analysis techniques.
This principle applies to real-valued functions defined on an appropriate Banach space of potentials, satisfying convexity, monotonicity, and translation invariance. Crucially, the framework admits applications to both the classical topological pressure for continuous maps and the topological pressure for semigroup actions. More precisely, they introduced the following abstract measure-theoretic entropy via the pressure function.

\begin{defn}
Let $f$ be a continuous self-map on a compact metric space $(X, d)$ with $\htop(f)<+\infty$ and $\mu$ be a Borel probability measure on $X$. The abstract measure-theoretic entropy of $f$ for $\mu$ is defined as
\[
\mathfrak{h}_\mu (f):=\inf_{\phi\in\mathcal{C}_f} \left\{\int_X \phi \, d\mu \right\},
\]
where
\[
\mathcal{C}_f:=\{\phi\in C(X):\,P(f,-\phi)\leq 0\}.
\]
\end{defn}

Based on this definition, Bi\'{s} et al.\cite{BCMV22} proved the following result.

\begin{thm}\cite[Theorem 5]{BCMV22}\label{thm:abstract-variational-principle-single}
Let $f$ be a continuous self-map on a compact metric space $(X, d)$ with $\htop(f)<+\infty$ and $\mu$ be a Borel probability measure on $X$. The abstract measure-theoretic entropy $\mathfrak{h}_\mu (f)$ satisfies:
\begin{enumerate}
  \item [(i)] $0\leq h_\mu(f)\leq \mathfrak{h}_\mu (f)$ for any $\mu\in \mathcal{P}_f(X)$.
  \item [(ii)] For every continuous potential $\varphi: X\to \R$,
  \begin{equation}\label{eq:abstract-variational-principle-single}
  \Ptop(f, \varphi)
  = \max_{\mu\in \mathcal{P}(X)} \left\{ \mathfrak{h}_\mu (f) + \int_X \varphi \, d\mu \right\}
  =\max_{\mu\in \mathcal{P}_f(X)} \left\{ \mathfrak{h}_\mu (f) + \int_X \varphi \, d\mu \right\}.
  \end{equation}
  \item[(iii)] Every measure $\mu\in \mathcal{P}(X)$ which attains the maximum \eqref{eq:abstract-variational-principle-single} is $f$-invariant.
\end{enumerate}
\end{thm}

Building upon the abstract measure-theoretic entropy framework for continuous maps, we introduce the following abstract measure-theoretic entropy for transition probability kernels.
\begin{defn}
Let $T$ be a correspondence on a compact metric space $(X, d)$, $\mathcal{Q}$ be a transition probability kernel on $X$ supported by $T$ and $\mu$ be a  Borel probability measure on $X$, define
\[
\mathfrak{h}_\mu (\mathcal{Q}):=\inf_{\phi\in\mathcal{C}_T} \left\{\int_X \int_{T(x_1)} \phi (x_1, x_2) \, d\mathcal{Q}_{x_1}(x_2)\, d\mu (x_1) \right\},
\]
where
\[
\mathcal{C}_T:=\{\phi\in C(\mathcal{O}_2(T)):\, \Ptop(T,-\phi)\leq 0\}.
\]
We call $\mathfrak{h}_\mu (\mathcal{Q})$ the {\it abstract measure-theoretic entropy of $\mathcal{Q}$ for $\mu$}.
\end{defn}

\begin{prop}
Let
\[
\mathcal{C}'_T=\{\phi\in C(\mathcal{O}_2(T)):\, \Ptop(T,-\phi)= 0\}.
\]
The abstract measure-theoretic entropy of $\mathcal{Q}$ for $\mu$
can also be defined as
\[
\mathfrak{h}_\mu (\mathcal{Q})=\inf_{\phi\in\mathcal{C}'_T} \left\{\int_X \int_{T(x_1)} \phi (x_1, x_2) \, d\mathcal{Q}_{x_1}(x_2)\, d\mu (x_1) \right\}.
\]
\end{prop}
\begin{proof}
Due to the definitions of $\mathcal{C}'_T$ and $\mathcal{C}_T$, it is suffices to prove that
\[
\mathfrak{h}_\mu (\mathcal{Q})\geq\inf_{\phi\in\mathcal{C}'_T} \left\{\int_X \int_{T(x_1)} \phi (x_1, x_2) \, d\mathcal{Q}_{x_1}(x_2)\, d\mu (x_1) \right\}.
\]
Given $\phi\in C(\mathcal{O}_2(T))$ with $\Ptop(T,-\phi)\leq 0$. By Lemma \ref{lem:basic properties of topological pressure for correspondences},
one has $$\Ptop(T,-\phi-\Ptop(T,-\phi))=0.$$
This indicates that
$\hat{\phi}:=\phi+\Ptop(T,-\phi)$ belongs to $\mathcal{C}'_T$. Hence,
\begin{align*}
\mathfrak{h}'_\mu (\mathcal{Q})&:=\inf_{\phi\in\mathcal{C}'_T} \left\{\int_X \int_{T(x_1)} \phi (x_1, x_2) \, d\mathcal{Q}_{x_1}(x_2)\, d\mu (x_1) \right\}\\
&\leq\int_X \int_{T(x_1)} \hat{\phi} (x_1, x_2) \, d\mathcal{Q}_{x_1}(x_2)\, d\mu (x_1)\\
&=\int_X \int_{T(x_1)} \phi (x_1, x_2) \, d\mathcal{Q}_{x_1}(x_2)\, d\mu (x_1)+\Ptop(T,-\phi)\\
&\leq\int_X \int_{T(x_1)} \phi (x_1, x_2) \, d\mathcal{Q}_{x_1}(x_2)\, d\mu (x_1),
\end{align*}
which immediately yields the desired inequality, as the function $\phi$ was chosen arbitrarily.
\end{proof}

\begin{rem}
The previous proposition demonstrates that for a transition probability kernel on a compact metric space supported by a correspondence $T$, its abstract measure-theoretic entropy can be completely characterized by potentials exhibiting vanishing topological pressure. Indeed, analogous results hold true for single-valued continuous maps as well (see also \cite{NWH22}).
\end{rem}

\begin{lem}\label{lem:auxiliary tool}
Let $T$ be a correspondence on a compact metric space $(X, d)$
and $\nu$ be a Borel probability measure on $\mathcal{O}_\omega(T)$.
If a transition probability kernel $\mathcal{Q}$ on $X$ supported by $T$
and a Borel probability measure $\mu$ on $X$ satisfy $\nu\circ\tilde{\pi}^{-1}_{12}=\mu\mathcal{Q}^{[1]}$, then $\mathfrak{h}_\nu (\sigma)\leq \mathfrak{h}_\mu (\mathcal{Q})$.
\end{lem}
\begin{proof}
According to the definitions of the abstract measure-theoretic entropies,
it suffices to prove that for any $\phi\in \mathcal{C}_T$, one can
find a function $\varphi\in \mathcal{C}_\sigma$ such that
\[
\int_{\mathcal{O}_\omega(T)}\varphi\, d\nu\leq\int_X \int_{T(x_1)} \phi (x_1, x_2) \, d\mathcal{Q}_{x_1}(x_2)\, d\mu (x_1).
\]
Given $\phi\in \mathcal{C}_T$, it follows from Lemma \ref{lem:connection of two pressure}
that
\[
\Ptop(\sigma,-\tilde{\phi})=\Ptop(T,-\phi)\leq0.
\]
Consequently, $\tilde{\phi}\in \mathcal{C}_\sigma$.
Moreover, by Lemma \ref{lem:a integral equality},
\begin{align*}
\int_{\mathcal{O}_\omega(T)}\tilde{\phi}\, d\nu
&=\int_{\mathcal{O}_2(T)}\phi\, d\nu\circ\tilde{\pi}^{-1}_{12}\\
&=\int_{\mathcal{O}_2(T)}\phi\, d(\mu\mathcal{Q}^{[1]})\\
&=\int_X \int_{T(x_1)} \phi (x_1, x_2) \, d\mathcal{Q}_{x_1}(x_2)\, d\mu (x_1).
\end{align*}
This ends the proof.
\end{proof}

We now establish the main result of this section: an abstract variational principle for the topological pressure of correspondences. A key innovation lies in deriving this principle without imposing additional conditions on the correspondences.

\begin{thm}\label{thm:type II variational principle}
Let $T$ be a correspondence on a compact metric space $(X, d)$.

\begin{enumerate}
    \item[(i)] For any continuous function $\phi: \mathcal{O}_2(T) \rightarrow \mathbb{R}$, the variational principle holds:
    \begin{equation}\label{eq:abstract-variational-principle-correspondence}
    \Ptop(T, \phi) = \max_{\mathcal{Q}, \mu} \left\{ \mathfrak{h}_\mu (\mathcal{Q}) + \int_X \int_{T(x_1)} \phi (x_1, x_2) \, d\mathcal{Q}_{x_1}(x_2) \, d\mu (x_1) \right\},
    \end{equation}
    where $\mathcal{Q}$ ranges over all transition probability kernels on $X$ supported by $T$, and $\mu$ ranges over all Borel probability measures or $\mathcal{Q}$-invariant Borel probability measures on $X$.

    \item[(ii)] Every pair $(\mathcal{Q}, \mu)$ which attains the maximal \eqref{eq:abstract-variational-principle-correspondence} satisfies $\mu\mathcal{Q}=\mu$. That is, $\mu$ is $\mathcal{Q}$-invariant.

    \item[(iii)] For any transition probability kernel $\mathcal{Q}$ on $X$ supported by $T$ and any Borel probability measure $\mu$ on $X$, the inverse variational principle holds:
    \[
    \mathfrak{h}_\mu (\mathcal{Q}) = \inf_{\phi} \left\{ \Ptop(T,\phi) - \int_X \int_{T(x_1)} \phi (x_1, x_2) \, d\mathcal{Q}_{x_1}(x_2) \, d\mu (x_1) \right\},
    \]
    where $\phi$ ranges over all continuous functions on $\mathcal{O}_2(T)$.

    \item[(iv)] $0\leq h_\mu (\mathcal{Q})\leq \mathfrak{h}_\mu (\mathcal{Q})$ for any transition probability kernel $\mathcal{Q}$ on $X$ supported by $T$ and any $\mu\in \mathcal{P}_\mathcal{Q}(X)$.
\end{enumerate}
\end{thm}

\begin{proof}
Fix a continuous function $\phi: \mathcal{O}_2(T) \rightarrow \mathbb{R}$.
Since $\Ptop(\phi-\Ptop(T,\phi))=0$, we know that $\hat{\phi}:=\Ptop(T,\phi)-\phi\in\mathcal{C}_T$. Hence, for any transition probability kernel $\mathcal{Q}$ on $X$ supported by $T$ and any Borel probability measure $\mu$ on $X$, it holds that
\begin{align*}
\mathfrak{h}_\mu (\mathcal{Q})&\leq\int_X \int_{T(x_1)} \hat{\phi} (x_1, x_2) \, d\mathcal{Q}_{x_1}(x_2)\, d\mu (x_1) \\
&=\Ptop(T,\phi)-\int_X \int_{T(x_1)} \phi (x_1, x_2) \, d\mathcal{Q}_{x_1}(x_2)\, d\mu (x_1).
\end{align*}
Consequently,
\[
\Ptop(T, \phi) \geq \sup_{\mathcal{Q}, \mu} \left\{ \mathfrak{h}_\mu (\mathcal{Q}) + \int_X \int_{T(x_1)} \phi (x_1, x_2) \, d\mathcal{Q}_{x_1}(x_2) \, d\mu (x_1) \right\},
\]
where $\mathcal{Q}$ ranges over all transition probability kernels on $X$ supported by $T$, and $\mu$ ranges over all Borel probability measures on $X$. Conversely, for any $\phi\in C(\mathcal{O}_2(T))$, we need to find
a transition probability kernel $\mathcal{Q}$ on $X$ supported by $T$ and a $\mathcal{Q}$-invariant Borel probability measure $\mu$ on $X$ such that
\[
\Ptop(T, \phi) \leq \mathfrak{h}_\mu (\mathcal{Q}) + \int_X \int_{T(x_1)} \phi (x_1, x_2) \, d\mathcal{Q}_{x_1}(x_2) \, d\mu (x_1).
\]
This can be done by means of the corresponding abstract variational principle for the shift map on the orbit space.
To this end, by Theorem \ref{thm:abstract-variational-principle-single}, there exists a $\sigma$-invariant Borel probability measure $\nu$ on
$\mathcal{O}_\omega(T)$ such that
\begin{align*}
\Ptop(\sigma, \tilde{\phi}) \leq \mathfrak{h}_\nu (\sigma) + \int_{\mathcal{O}_\omega(T)}  \tilde{\phi}\, d\nu.
\end{align*}
According to Lemma \cite[Lemma 6.16]{LLZ23}, we can find a transition probability kernel $\mathcal{Q}$ on $X$ supported by $T$ and a $\mathcal{Q}$-invariant Borel probability measure on $X$ satisfying $\nu\circ\tilde{\pi}^{-1}_{12}=\mu\mathcal{Q}^{[1]}$, which together with Lemma \ref{lem:connection of two pressure}, Lemma \ref{lem:a integral equality}, and Lemma \ref{lem:auxiliary tool} implies that
\begin{align*}
\Ptop(T, \phi)=\Ptop(\sigma, \tilde{\phi}) &\leq \mathfrak{h}_\nu (\sigma) + \int_{\mathcal{O}_\omega(T)}  \tilde{\phi}\, d\nu\\
&\leq\mathfrak{h}_\mu (\mathcal{Q}) + \int_{\mathcal{O}_2(T)}\phi\, d(\mu\mathcal{Q}^{[1]})\\
&=\mathfrak{h}_\mu (\mathcal{Q}) +\int_X \int_{T(x_1)} \phi (x_1, x_2) \, d\mathcal{Q}_{x_1}(x_2)\, d\mu (x_1).
\end{align*}
Thus, we have finished the proof of assertion (i).

Suppose $\mathcal{Q}_\phi$ is a transition probability kernels on $X$ supported by $T$, and $\mu_\phi$ a Borel probability measure on $X$ such that the pair $(\mathcal{Q}_\phi,\mu_\phi)$ satisfies
\[
\Ptop(T, \phi)=\mathfrak{h}_\mu (\mathcal{Q}_\phi) +\int_X \int_{T(x_1)} \phi (x_1, x_2) \, d({\mathcal{Q}_\phi})_{x_1}(x_2)\, d\mu_\phi (x_1).
\]
For $\psi\in C(X)$. Combining Lemma \ref{lem:basic properties of topological pressure for correspondences}, Lemma \ref{lem:a integral equality}
and statement (i) yields
\begin{align*}
\mathfrak{h}_\mu (\mathcal{Q}_\phi) +\int_{\mathcal{O}_2(T)}\phi\, d(\mu_\phi\mathcal{Q}_\phi^{[1]})
&=\Ptop(T, \phi)\\
&=\Ptop(T, \phi+\psi\circ \tilde{\pi}_1-\psi\circ \tilde{\pi}_2)\\
&\geq \mathfrak{h}_\mu (\mathcal{Q}_\phi) +\int_{\mathcal{O}_2(T)}(\phi+\psi\circ \tilde{\pi}_1-\psi\circ \tilde{\pi}_2)\, d(\mu_\phi\mathcal{Q}_\phi^{[1]}).
\end{align*}
Hence,
\[
\int_{\mathcal{O}_2(T)}\psi\circ \tilde{\pi}_1\, d(\mu_\phi\mathcal{Q}_\phi^{[1]})
\leq\int_{\mathcal{O}_2(T)}\psi\circ \tilde{\pi}_2\, d(\mu_\phi\mathcal{Q}_\phi^{[1]}).
\]
Furthermore, since (see \cite[Corollarys A.4, A.7 and Lemma 6.13]{LLZ23})
\[
(\mu_\phi\mathcal{Q}_\phi^{[1]})\circ \tilde{\pi}_1^{-1}=\mu_\phi,~
(\mu_\phi\mathcal{Q}_\phi^{[2]})\circ \tilde{\pi}_2^{-1}=\mu_\phi\mathcal{Q}_\phi,
~\text{ and }(\mu_\phi\mathcal{Q}_\phi^{[1]})(\mathcal{O}_2(T))=1,
\]
we obtain
\begin{align*}
\int_X\psi\,d\mu_\phi&=\int_{X^2}\psi\circ \tilde{\pi}_1\, d(\mu_\phi\mathcal{Q}_\phi^{[1]})\\
&=\int_{\mathcal{O}_2(T)}\psi\circ \tilde{\pi}_1\, d(\mu_\phi\mathcal{Q}_\phi^{[1]})\\
&\leq\int_{\mathcal{O}_2(T)}\psi\circ \tilde{\pi}_2\, d(\mu_\phi\mathcal{Q}_\phi^{[1]})\\
&=\int_{X^2}\psi\circ \tilde{\pi}_2\, d(\mu_\phi\mathcal{Q}_\phi^{[1]})\\
&=\int_X\psi\,d(\mu_\phi\mathcal{Q}_\phi).
\end{align*}
Therefore, the inequality
\[
\int_X\psi\,d\mu_\phi\leq \int_X\psi\,d(\mu_\phi\mathcal{Q}_\phi)
\]
holds for all $\psi\in C(X)$. Applying this with $-\psi$ gives the reverse inequality.
Therefore, we get
\[
\int_X\psi\,d\mu_\phi= \int_X\psi\,d(\mu_\phi\mathcal{Q}_\phi),
\]
which implies $\mu_\phi=\mu_\phi\mathcal{Q}_\phi$. Consequently, $\mu_\phi$ is $\mathcal{Q}_\phi$-invariant. This proves (ii).

Next, we will continue with the proof of assertion (iii). For any transition probability kernel $\mathcal{Q}$ on $X$ supported by $T$ and any Borel probability measure $\mu$ on $X$. By statement (i) we get
\[
\mathfrak{h}_\mu (\mathcal{Q}) \leq P(T,\phi) - \int_X \int_{T(x_1)} \phi (x_1, x_2) \, d\mathcal{Q}_{x_1}(x_2) \, d\mu (x_1)
\]
for any $\phi\in C(\mathcal{O}_2(T))$. Therefore,
\begin{align*}
\mathfrak{h}_\mu (\mathcal{Q})
&\leq\inf_{\phi\in C(\mathcal{O}_2(T))} \left\{ \Ptop(T,-\phi) + \int_X \int_{T(x_1)} \phi (x_1, x_2) \, d\mathcal{Q}_{x_1}(x_2) \, d\mu (x_1) \right\}\\
&\leq\inf_{\phi\in \mathcal{C}_T} \left\{ \Ptop(T,-\phi) + \int_X \int_{T(x_1)} \phi (x_1, x_2) \, d\mathcal{Q}_{x_1}(x_2) \, d\mu (x_1) \right\}\\
&\leq\inf_{\phi\in \mathcal{C}_T} \left\{ \int_X \int_{T(x_1)} \phi (x_1, x_2) \, d\mathcal{Q}_{x_1}(x_2) \, d\mu (x_1) \right\}\\
&=\mathfrak{h}_\mu (\mathcal{Q}),
\end{align*}
which immediately implies that
\begin{align*}
\mathfrak{h}_\mu (\mathcal{Q})
&=\inf_{\phi\in C(\mathcal{O}_2(T))} \left\{ \Ptop(T,-\phi) + \int_X \int_{T(x_1)} \phi (x_1, x_2) \, d\mathcal{Q}_{x_1}(x_2) \, d\mu (x_1) \right\}\\
&=\inf_{\phi\in C(\mathcal{O}_2(T))} \left\{ \Ptop(T,\phi) - \int_X \int_{T(x_1)} \phi (x_1, x_2) \, d\mathcal{Q}_{x_1}(x_2) \, d\mu (x_1) \right\}.
\end{align*}
This ends the proof.

(iv) follows immediately from (iii) and \cite[Theorem D]{LLZ23}.
\end{proof}

\begin{defn}\label{defn:abstract measure entropy for correspondences}
Let $T$ be a correspondence on a compact metric space $(X, d)$, $\mu$ be a $T$-invariant Borel probability measure and $\phi: \mathcal{O}_2(T) \rightarrow \mathbb{R}$ be a continuous function. We define the {\it abstract measure-theoretic pressure of $\varphi$ for $\mu$} and the {\it abstract measure-theoretic entropy of $\mu$} as follows:
\[
\mathfrak{P}_\mu (T,\phi)=\sup_{\mathcal{Q}\in \mathcal{K}_\mu}
\left\{\mathfrak{h}_\mu (\mathcal{Q})+ \int_X \int_{T(x_1)} \phi (x_1, x_2) \, d\mathcal{Q}_{x_1}(x_2) \, d\mu (x_1)\right\},
\]
and
\[
\mathfrak{h}_\mu (T)=\sup_{\mathcal{Q}\in \mathcal{K}_\mu}\{\mathfrak{h}_\mu (\mathcal{Q})\}.
\]
\end{defn}

Theorem \ref{thm:type II variational principle} can be reformulated within the framework of Definition \ref{defn:abstract measure entropy for correspondences} as follows:
\begin{thm}
Let $T$ be a correspondence on a compact metric space $(X, d)$ and
$\phi: \mathcal{O}_2(T) \rightarrow \mathbb{R}$ be a continuous function. Then we have
\[
\Ptop(T, \phi) =\max_{\mu\in \mathcal{P}_T(X)} \left\{ \mathfrak{P}_\mu (T,\phi) \right\}.
\]
Especially,
\[
\htop(T) = \max_{\mu\in\mathcal{P}_T(X)} \left\{ \mathfrak{h}_\mu (T) \right\}.
\]
\end{thm}

Next we explore the behavior of the abstract measure-theoretic pressure (entropy) under topological conjugacy.
\begin{thm}
Let $T$ be a correspondence on a compact metric space $X$, let $S$ be a correspondence on a compact metric space $Y$, let $\mu$ be a $T$-invariant measure, and let $\phi: \mathcal{O}_2(S) \rightarrow \mathbb{R}$ be a continuous function. If $T$ and $S$ are topologically conjugate via a homeomorphism $\theta: X\to Y$, then $\mu\circ\theta^{-1}$ is an $S$-invariant measure and
\[
\mathfrak{P}_\mu(T,\varphi)=\mathfrak{P}_{\mu\circ\theta^{-1}}(S,\phi),
\]
where $\varphi:=\phi\circ\theta^{(2)}|_{\mathcal{O}_2(T)}$. Especially,
\[
\mathfrak{h}_\mu(T)=\mathfrak{h}_{\mu\circ\theta^{-1}}(S).
\]
\end{thm}
\begin{proof}
It follows from Theorem \ref{thm:measure pressure under topological conjugacy} that $\mu\circ\theta^{-1}$ is $S$-invariant and $\varphi$ is well-defined.

By Theorems \ref{thm:topological pressure under topological conjugacy} and \ref{thm:measure pressure under topological conjugacy}, we have the following claim.

\textbf{Claim.} Let $\psi'\in C(\mathcal{O}_2(T))$ and $\psi\in C(\mathcal{O}_2(S))$
be continuous functions such that $\psi'=\psi\circ\theta^{(2)}|_{\mathcal{O}_2(T)}$. Then we have
\[
\Ptop(T,\psi')=\Ptop(S,\psi).
\]
Furthermore, let $\mathcal{Q}$ be a transition probability kernel on $X$ and $\mathcal{L}$ a transition probability kernel on $Y$. If
\[
\mathcal{Q}(x,A)=\mathcal{L}(\theta(x),\theta(A))~
\text{ for all }x\in X \text{ and } A\in \mathscr{B}(X),
\]
then
\[
\int_X \int_{T(x_1)} \psi' (x_1, x_2) \, d\mathcal{Q}_{x_1}(x_2) \, d\mu (x_1)
=\int_Y \int_{S(y_1)} \psi (y_1, y_2) \, d\mathcal{L}_{y_1}(y_2)
\, d(\mu\circ\theta^{-1}) (x_1).
\]

Note that since $\mu$ is $T$-invariant, there exists a transition probability kernel on $X$ supported by $T$ such that $\mu\mathcal{Q}=\mu$.
Let $\mathcal{L}: Y\times \mathscr{B}(Y)\to [0,1]$ be defined as follow:
for any $y\in Y$ and $B\in \mathscr{B}(Y)$, set
\[
\mathcal{L}(y,B):=\mathcal{Q}(\theta^{-1}(y),\theta^{-1}(B)).
\]
It follows from Theorem \ref{thm:measure pressure under topological conjugacy} that $\mathcal{L}$ is a transition probability kernel on $Y$ supported by $S$ and that $(\mu\circ\theta^{-1})$ is $\mathcal{L}$-invariant.

Therefore, combining Theorem \ref{thm:type II variational principle} and the preceding claim, we obtain
\begin{align*}
\mathfrak{h}_\mu (\mathcal{Q}) &= \inf_{\psi'\in C(\mathcal{O}_2(T))} \left\{ \Ptop(T,\psi') - \int_X \int_{T(x_1)} \psi' (x_1, x_2) \, d\mathcal{Q}_{x_1}(x_2) \, d\mu (x_1) \right\}\\
&=\inf_{\psi\in C(\mathcal{O}_2(S))} \left\{ \Ptop(S,\psi) - \int_Y \int_{S(x_1)} \psi (y_1, y_2) \, d\mathcal{L}_{y_1}(y_2) \, d(\mu\circ\theta^{-1}) (y_1) \right\}\\
&=\mathfrak{h}_{\mu\circ\theta^{-1}} (\mathcal{L}).
\end{align*}
Hence,
\begin{align*}
\mathfrak{P}_\mu (T,\varphi)&=\sup_{\mathcal{Q}\in \mathcal{K}_\mu}
\left\{\mathfrak{h}_\mu (\mathcal{Q})+ \int_X \int_{T(x_1)} \varphi (x_1, x_2) \, d\mathcal{Q}_{x_1}(x_2) \, d\mu (x_1)\right\}\\
&\leq\sup_{\mathcal{L}\in \mathcal{K}_{\mu\circ\theta^{-1}}}
\left\{\mathfrak{h}_{\mu\circ\theta^{-1}}(\mathcal{L})+ \int_X \int_{T(x_1)} \phi (x_1, x_2) \, d\mathcal{L}_{x_1}(x_2) \, d(\mu\circ\theta^{-1}) (x_1)\right\}\\
&=\mathfrak{P}_{\mu\circ\theta^{-1}} (S,\phi).
\end{align*}
Moreover, by symmetry of the conjugacy $\theta$, it holds that
$\mathfrak{P}_\mu(T,\varphi)=\mathfrak{P}_{\mu\circ\theta^{-1}}(S,\phi)$.
\end{proof}

\section{Differentiability and equilibrium states of the topological pressure}

This section focuses on the differentiability of the topological pressure and its associated equilibrium states for correspondences.

\subsection{Differentiability of the topological pressure}
In this subsection, we analyze the differentiability of the topological pressure for correspondences. Let $(X, d)$ be a compact metric space
and $T$ be a correspondence on $(X, d)$.
For any pair of continuous potential functions $\phi, \varphi: \mathcal{O}_2(T) \rightarrow \mathbb{R}$, the convexity property of topological pressure guarantees the existence of the following limits:
\[
d^+\Ptop(T,\phi)(\varphi):=\lim_{t\to0^+}\frac{1}{t}(\Ptop(T,\phi+t\varphi)-\Ptop(T,\phi)),
\]
\[
d^-\Ptop(T,\phi)(\varphi):=\lim_{t\to0^-}\frac{1}{t}(\Ptop(T,\phi+t\varphi)-\Ptop(T,\phi)).
\]
It can be demonstrated that
\begin{enumerate}
  \item[(i)] $d^-\Ptop(T,\phi)(\varphi)=-d^+\Ptop(T,\phi)(-\varphi)$;
  \item[(ii)] $d^-\Ptop(T,\phi)(\varphi)\leq d^+\Ptop(T,\phi)(\varphi)$;
  \item[(iii)] $d^+\Ptop(T,\phi)(\lambda\varphi)= \lambda d^+\Ptop(T,\phi)(\varphi)$ for $\lambda\geq0$;
  \item[(iv)] $d^+\Ptop(T,\phi)(\varphi_1+\varphi_2)\leq d^+\Ptop(T,\phi)(\varphi_1)+d^+\Ptop(T,\phi)(\varphi_2)$.
\end{enumerate}

\begin{defn}
Let $T$ be a correspondence on a compact metric space $(X, d)$ and $\phi: \mathcal{O}_2(T) \rightarrow \mathbb{R}$ be a continuous function.
We call the topological pressure $\Ptop(T,\cdot)$ of $T$ {\it Gateaux differentiable at $\phi$}
if
\[
d\Ptop(T,\phi)(\varphi):=\lim_{t\to0}\frac{1}{t}(\Ptop(T,\phi+t\varphi)-\Ptop(T,\phi))
\]
exists for all $\varphi\in C(\mathcal{O}_2(T))$.
\end{defn}

\begin{rem}
It is obvious to see that the topological pressure of $T$
is Gateaux differentiable at $\phi$ if and only if the following equivalent conditions hold:
\begin{enumerate}
  \item[(i)] the function $t\mapsto \Ptop(T,\phi+t\varphi)$ is differentiable at $t=0$ for all $\varphi\in C(\mathcal{O}_2(T))$;
  \item[(ii)] $d^+\Ptop(T,\phi)(\varphi)=-d^+\Ptop(T,\phi)(-\varphi)$ for all $\varphi\in C(\mathcal{O}_2(T))$;
  \item[(iii)] the functional $\varphi\mapsto d^+\Ptop(T,\phi)(\varphi)$ is linear.
\end{enumerate}
\end{rem}

Next we introduce a related notion.

\begin{defn}
Let $T$ be a correspondence on a compact metric space $(X, d)$ satisfying $\htop(T)<+\infty$ and $\phi: \mathcal{O}_2(T) \rightarrow \mathbb{R}$
be a continuous function.
Let \(\mathcal{Q}\) be a transition probability kernel on \((X, \mathscr{B}(X))\) supported by $T$ and $\mu$ be a Borel probability measure on $(X,\mathscr{B}(X))$.
We call the pair $(\mathcal{Q},\mu)$ a {\it tangent functional} to $\Ptop(T,\cdot)$ at $\phi$ if
\[
\Ptop(T,\phi+\varphi)-\Ptop(T,\phi)\geq\int_{\mathcal{O}_2(T)}\varphi\, d(\mu\mathcal{Q}^{[1]}) \text{ for any } \varphi\in C(\mathcal{O}_2(T)).
\]
Denote by $t_\phi(X,T)$ the collection of all tangent functionals to $\Ptop(T,\cdot)$ at $\phi$.
\end{defn}

Some basic properties of the tangent functionals are collected as follows.
\begin{lem}\label{lem:property of tangent functional}
Let $T$ be a correspondence on a compact metric space $(X, d)$ satisfying $\htop(T)<+\infty$ and $\phi, \varphi: \mathcal{O}_2(T) \rightarrow \mathbb{R}$ be continuous functions.
The following statements hold.
\begin{enumerate}
  \item[(i)] $t_\phi(X,T)\neq\emptyset$;
  \item[(ii)] $(\mathcal{Q},\mu)\in t_\phi(X,T)$ if and only if
  \[
  \Ptop(T,\phi)-\int_{\mathcal{O}_2(T)}\phi\, d(\mu\mathcal{Q}^{[1]})=\inf\left\{\Ptop(T,\psi)-\int_{\mathcal{O}_2(T)}\psi\, d(\mu\mathcal{Q}^{[1]}):\psi\in C(\mathcal{O}_2(T))\right\};
  \]
  \item[(iii)] $d^+\Ptop(T,\phi)(\varphi)=\max\left\{\int_{\mathcal{O}_2(T)}\varphi\, d(\mu\mathcal{Q}^{[1]}):(\mathcal{Q},\mu)\in t_\phi(X,T)\right\}$;
  \item[(iv)] $d^-\Ptop(T,\phi)(\varphi)=\min\left\{\int_{\mathcal{O}_2(T)}\varphi\, d(\mu\mathcal{Q}^{[1]}):(\mathcal{Q},\mu)\in t_\phi(X,T)\right\}$.
\end{enumerate}
\end{lem}
\begin{proof}
(i) By the convexity of the continuous linear functional $\psi\mapsto\Ptop(T,\phi+\psi)-\Ptop(T,\phi)$ and the Hahn-Banach theorem, one can find a continuous linear functional $L$ on $C(\mathcal{O}_2(T))$ such that
\[
L(\psi)\leq\Ptop(T,\phi+\psi)-\Ptop(T,\phi)\text{ for any }
\psi\in C(\mathcal{O}_2(T)).
\]
Further applying Riesz representation theorem yields that there exists
a finite signed measure $\nu$ on $\mathcal{O}_2(T)$ such that
\[
L(\psi)=\int_{\mathcal{O}_2(T)}\psi\, d\nu \text{ for any }
\psi\in C(\mathcal{O}_2(T)).
\]

\textbf{Claim.} We prove that $\nu$ is a Borel probability measure on $\mathcal{O}_2(T)$.

For any $\psi\in C(\mathcal{O}_2(T))$ with $\psi\geq0$ and $\epsilon>0$, Lemma \ref{lem:basic properties of topological pressure for correspondences} implies that
\begin{align*}
\int_{X^2}(\psi+\ep)\, d\nu&=L(\psi+\ep)\\
&=-L(-(\psi+\ep))\\
&\geq-(\Ptop(T,\phi-(\psi+\ep)))+\Ptop(T,\phi)\\
&\geq-(\Ptop(T,\phi)-\inf(\psi+\ep))+\Ptop(T,\phi)\\
&\geq\inf(\psi+\ep)\\
&>0.
\end{align*}
Hence, $\nu$ is a finite measure on $\mathcal{O}_2(T)$. Meanwhile, for any $n\in\Z$, we have
\[
n\nu(\mathcal{O}_2(T))=\int_{\mathcal{O}_2(T)}n\, d\nu\leq\Ptop(T,\phi+n)-\Ptop(T,\phi)=n,
\]
which immediately implies $\nu(\mathcal{O}_2(T))=1$. So the claim is true.

By \cite[Proposition A.11]{LLZ23} we can find a Borel probability measure $\mu$ on $X$ and a transition probability kernel \(\mathcal{Q}\) on $X$ supported by $T$ such that $\mu\mathcal{Q}^{[1]}=\nu$. As a result, we have
\[
\int_{\mathcal{O}_2(T)}\psi\, d(\mu\mathcal{Q}^{[1]})\leq\Ptop(T,\phi+\psi)-\Ptop(T,\phi)\text{ for any }
\psi\in C(\mathcal{O}_2(T)).
\]
So $(\mathcal{Q},\mu)\in t_\phi(X,T)$.

(ii) If $(\mathcal{Q},\mu)\in t_\phi(X,T)$, then for any $\psi\in C(\mathcal{O}_2(T))$,
\begin{align*}
\Ptop(T,\phi)-\int_{\mathcal{O}_2(T)}\phi\, d(\mu\mathcal{Q}^{[1]})
&\leq\Ptop(T,\psi)-\int_{\mathcal{O}_2(T)}(\psi-\phi)\, d(\mu\mathcal{Q}^{[1]})-\int_{\mathcal{O}_2(T)}\phi\, d(\mu\mathcal{Q}^{[1]})\\
&=\Ptop(T,\psi)-\int_{\mathcal{O}_2(T)}\psi\, d(\mu\mathcal{Q}^{[1]}),
\end{align*}
which clearly implies the necessity.
Conversely, if
\[
\Ptop(T,\phi)-\int_{\mathcal{O}_2(T)}\phi\, d(\mu\mathcal{Q}^{[1]})=\inf\left\{\Ptop(T,\psi)-\int_{\mathcal{O}_2(T)}\psi\, d(\mu\mathcal{Q}^{[1]}):\psi\in C(\mathcal{O}_2(T))\right\},
\]
then for any $\varphi\in C(\mathcal{O}_2(T))$ we have
\[
\Ptop(T,\phi)-\int_{\mathcal{O}_2(T)}\phi\, d(\mu\mathcal{Q}^{[1]})
\leq\Ptop(T,\phi+\varphi)-\int_{\mathcal{O}_2(T)}(\phi+\varphi)\, d(\mu\mathcal{Q}^{[1]}).
\]
So
\[
\Ptop(T,\phi+\varphi)-\Ptop(T,\phi)\geq\int_{\mathcal{O}_2(T)}\varphi\, d(\mu\mathcal{Q}^{[1]}),
\]
which implies that $(\mathcal{Q},\mu)\in t_\phi(X,T)$.

(iii) If $(\mathcal{Q},\mu)\in t_\phi(X,T)$ then
\[
\int_{\mathcal{O}_2(T)}\varphi\, d(\mu\mathcal{Q}^{[1]})\leq\frac{1}{t}
\left(\Ptop(T,\phi+t\varphi)-\Ptop(T,\phi)\right)
\]
for any $t>0$. Letting $t\to0^+$ gives $\int_{\mathcal{O}_2(T)}\varphi\, d(\mu\mathcal{Q}^{[1]})\leq d^+\Ptop(T,\phi)(\varphi)$.

On the other hand, let $a=d^+\Ptop(T,\phi)(\varphi)$ and consider a linear functional $\gamma$ on $\{t\varphi:t\in\R\}$ defined by $\gamma(t\varphi):=ta$. By the property of $d^+\Ptop(T,\phi)(\varphi)$
we have $\gamma(t\varphi)=td^+\Ptop(T,\phi)(\varphi)
\leq\Ptop(T,\phi+t\varphi)-\Ptop(T,\phi)$. Now adopting a procedure similar to the proof of statement (i), we can choose a transition probability kernel \(\mathcal{Q}\) on $X$ supported by $T$ and a Borel probability measure $\mu$ on $X$ such that $(\mathcal{Q},\mu)\in t_\phi(X,T)$
and
\[
\int_{\mathcal{O}_2(T)}\varphi\, d(\mu\mathcal{Q}^{[1]})
=\gamma(\varphi)=d^+\Ptop(T,\phi)(\varphi).
\]
This ends the proof of statement (iii).

(iv) is a consequence of (iii) and the fact that $d^-\Ptop(T,\phi)(\varphi)=-d^+\Ptop(T,\phi)(-\varphi)$.
\end{proof}

Next, we present a result concerning the differentiability of the topological pressure for correspondences. The following theorem will introduce some new concepts. As these concepts are not directly utilized in our work, no specific definitions are provided in the context. Readers may refer to \cite{LLZ23} for their definitions.

\begin{thm}\label{thm:differentiability of pressure}
Let $T$ be a correspondence on a compact metric space $(X, d)$ satisfying $\htop(T)<+\infty$. The topological pressure of $T$ is Gateaux differentiable at $\phi\in C(\mathcal{O}_2(T))$ if and only if there is a unique tangent functional $(\mathcal{Q},\mu)$ to $\Ptop(T,\cdot)$ at $\phi$ in the sense that the measure $\mu$ is unique and that if there are two tangent functionals $(\mathcal{Q},\mu)$ and $(\mathcal{Q}',\mu)$, then for $\mu$-almost every $x\in X$ and all $A\in \mathscr{B}(X)$, the equality $\mathcal{Q}_x(A)=\mathcal{Q}'_x(A)$ holds.

Moreover, if
one of the following two conditions holds:
\begin{enumerate}
  \item[(i)] $T$ is a forward expansive correspondence with the specification property and
  $\phi, \varphi\in C(\mathcal{O}_2(T))$ are Bowen summable continuous functions;
  \item[(ii)] $T$ is an open, strongly transitive, distance-expanding correspondence on $X$ and
  $\phi, \varphi\in C(\mathcal{O}_2(T))$ are H\"{o}lder continuous functions,
\end{enumerate}
then the topological pressure of $T$ is Gateaux differentiable at $\phi$ and
\[
d\Ptop(T,\phi)(\varphi)
=\int_X \int_{T(x_1)} \varphi (x_1, x_2) \,
d(\mathcal{Q}_\phi)_{x_1}(x_2)\, d\mu_\phi (x_1),
\]
where $(\mathcal{Q}_\phi,\mu_\phi)$ is the unique tangent functional to $\Ptop(T,\cdot)$ at $\phi$. Besides, the function $t\mapsto \Ptop(T,\phi+t\varphi)$ is differentiable on $\R$ and
\[
\frac{d}{dt}\Ptop(T,\phi+t\varphi)=\int_X \int_{T(x_1)} \varphi (x_1, x_2) \,
d(\mathcal{Q}_t)_{x_1}(x_2)\, d\mu_t (x_1),
\]
where $(\mathcal{Q}_t,\mu_t)$ is the unique tangent functional to $\Ptop(T,\cdot)$ at $\phi+t\varphi$.

\end{thm}
\begin{proof}
Assume that $t_\phi(X,T)$ is unique in the sense defined above. If
$(\mathcal{Q},\mu), (\mathcal{Q}',\mu)\in t_\phi(X,T)$ then
$\mu\mathcal{Q}^{[1]}=\mu\mathcal{Q}'^{[1]}$.
By Lemma \ref{lem:property of tangent functional}, we have
\[
d^+\Ptop(T,\phi)(\varphi)=d^-\Ptop(T,\phi)(\varphi).
\]
So the topological pressure of $T$
is Gateaux differentiable at $\phi$.

If the set $t_\phi(X,T)$ is not unique, then we can choose two pairs $(\mathcal{Q},\mu), (\mathcal{Q}',\mu')$ in $t_\phi(X,T)$. \cite[Proposition A.11]{LLZ23} guarantees that
$\mu\mathcal{Q}^{[1]}\neq\mu'\mathcal{Q}'^{[1]}$. Then there must be a continuous function
$\varphi\in C(\mathcal{O}_2(T))$ such that $\int_{\mathcal{O}_2(T)}\phi\, d(\mu\mathcal{Q}^{[1]})\neq\int_{\mathcal{O}_2(T)}\phi\, d(\mu'\mathcal{Q}'^{[1]})$.
Now applying Lemma \ref{lem:property of tangent functional} gives
$d^+\Ptop(T,\phi)(\varphi)>d^-\Ptop(T,\phi)(\varphi)$. Thus, the topological pressure of $T$
is not Gateaux differentiable at $\phi$.

Let $T$ be a forward expansive correspondence with the specification property and
$\phi, \varphi\in C(\mathcal{O}_2(T))$ be Bowen summable continuous functions.
According to the discussion of \cite{LLZ23}, the dynamical system $(\mathcal{O}_\omega(T),\sigma)$ is forward expansive and has the specification property, and the continuous functions $\tilde{\phi}, \tilde{\varphi}: \mathcal{O}_\omega(T)\to \R$ are Bowen summable with respect to $\sigma$, where
\[
\tilde{\phi}(x_1,x_2,\ldots)=\phi(x_1,x_2)\text{ and }
\tilde{\varphi}(x_1,x_2,\ldots)=\varphi(x_1,x_2).
\]
By \cite[Proposition 7.10]{LLZ23} and \cite[Corollary 2]{Wa92}, the function $t\mapsto \Ptop(\sigma,g+th)$ is differentiable at $t=0$ for those $g,h\in C(\mathcal{O}_\omega(T))$
which are Bowen summable with respect to $\sigma$. Therefore, the following limit exists:
\begin{align*}
d\Ptop(T,\phi)(\varphi)&=
\lim_{ t\to0}\frac{1}{ t}(\Ptop(T,\phi+t\varphi)-\Ptop(T,\phi))\\
&=\lim_{ t\to0}\frac{1}{ t}(\Ptop(\sigma,\tilde{\phi}+t\tilde{\varphi})
-\Ptop(\sigma,\tilde{\phi})).
\end{align*}
So the topological pressure of $T$ is Gateaux differentiable at $\phi$.
Moreover, based on previous discussion, the set $t_\phi(X,T)$ is unique in the sense defined above. Denote by $(\mathcal{Q}_\phi,\mu_\phi)$ the unique tangent functional to $\Ptop(T,\cdot)$ at $\phi$. By Lemma \ref{lem:property of tangent functional} and Lemma Lemma \ref{lem:a integral equality} we have
\[
d\Ptop(T,\phi)(\varphi)=\int_{\mathcal{O}_2(T)}\varphi\, d(\mu_\phi\mathcal{Q}_\phi^{[1]})
=\int_X \int_{T(x_1)} \varphi (x_1, x_2) \,
d(\mathcal{Q}_\phi)_{x_1}(x_2)\, d\mu_\phi (x_1).
\]

Meanwhile, for any $t\in\R$, we can obtain
\begin{align*}
\frac{d}{dt}\Ptop(T,\phi+t\varphi)&=
\lim_{\Delta t\to0}\frac{1}{\Delta t}(\Ptop(T,\phi+(t+\Delta t)\varphi)-\Ptop(T,\phi+t\varphi))\\
&=d\Ptop(T,\phi+t\varphi)(\varphi)\\
&=\int_X \int_{T(x_1)} \varphi (x_1, x_2) \,
d(\mathcal{Q}_t)_{x_1}(x_2)\, d\mu_t (x_1),
\end{align*}
where $(\mathcal{Q}_t,\mu_t)$ is the unique tangent functional to $\Ptop(T,\cdot)$ at $\phi+t\varphi$.
Hence, the function $t\mapsto \Ptop(T,\phi+t\varphi)$ is differentiable on $\R$.

Another case can be established by combining \cite[Proposition 7.15]{LLZ23} and similar techniques.
\end{proof}

\subsection{Equilibrium states}

In this section, we investigate two types of equilibrium states for correspondences.

\begin{defn}
Let $T$ be a correspondence on a compact metric space $(X, d)$ and
$\phi: \mathcal{O}_2(T) \rightarrow \mathbb{R}$ be a continuous function.

\begin{enumerate}
    \item[(i)] Let \(\mathcal{Q}\) be a transition probability kernel on \((X, \mathscr{B}(X))\) and $\mu$ be a \(\mathcal{Q}\)-invariant Borel probability measure on $(X,\mathscr{B}(X))$. We call the pair $(\mathcal{Q},\mu)$ a {\it type I equilibrium state} for the correspondence $T$ and the potential function $\phi$ if
    \[
    \Ptop(T, \phi) = h_\mu (\mathcal{Q}) + \int_X \int_{T(x_1)} \phi (x_1, x_2) \, d\mathcal{Q}_{x_1}(x_2) \, d\mu (x_1).
    \]
    Denote by $e^1_\phi(X,T)$ the set of type I equilibrium states for the correspondence $T$ and the potential function $\phi$.

    \item[(ii)] Let \(\mathcal{Q}\) be a transition probability kernel on \((X, \mathscr{B}(X))\) and $\mu$ be a Borel probability measure on $(X,\mathscr{B}(X))$. We call the pair $(\mathcal{Q},\mu)$ a {\it type II equilibrium state} for the correspondence $T$ and the potential function $\phi$ if
    \[
    \Ptop(T, \phi) = \mathfrak{h}_\mu (\mathcal{Q}) + \int_X \int_{T(x_1)} \phi (x_1, x_2) \, d\mathcal{Q}_{x_1}(x_2) \, d\mu (x_1).
    \]
    Denote by $e^2_\phi(X,T)$ the set of type II equilibrium states for the correspondence $T$ and the potential function $\phi$.
\end{enumerate}
\end{defn}

\begin{rem}
Some remarks are in order.
\begin{enumerate}
  \item[(i)] It should be noted that the first type of equilibrium state is defined for \(\mathcal{Q}\)-invariant Borel probability measures, while the second type applies to general Borel probability measures.
  \item[(ii)] By \cite[Theorem D]{LLZ23}, we obtain that $(\mathcal{Q},\mu)\in e^1_\phi(X,T)$ if and only if
\[
\Ptop(T, \phi) = \sup_{\mathcal{Q}, \mu} \left\{ h_\mu (\mathcal{Q}) + \int_X \int_{T(x_1)} \phi (x_1, x_2) \, d\mathcal{Q}_{x_1}(x_2) \, d\mu (x_1) \right\},
\]
where $\mathcal{Q}$ ranges over all transition probability kernels on $X$ supported by $T$, and $\mu$ ranges over all $\mathcal{Q}$-invariant Borel probability measures on $X$. Moreover, by Theorem \ref{thm:type II variational principle},
$(\mathcal{Q},\mu)\in e^2_\phi(X,T)$ if and only if
\[
\Ptop(T, \phi) = \max_{\mathcal{Q}, \mu} \left\{ \mathfrak{h}_\mu (\mathcal{Q}) + \int_X \int_{T(x_1)} \phi (x_1, x_2) \, d\mathcal{Q}_{x_1}(x_2) \, d\mu (x_1) \right\},
\]
where $\mathcal{Q}$ ranges over all transition probability kernels on $X$ supported by $T$, and $\mu$ ranges over all Borel probability measures on $X$.
  \item [(iii)] It follows from Theorem \ref{thm:type II variational principle} that the set $e^2_\phi(X,T)$ is nonempty.
\end{enumerate}

\end{rem}

Now we explore the connection between equilibrium states and tangent functionals.

\begin{thm}\label{thm:characterization of equilibrium states}
Let $T$ be a correspondence on a compact metric space $(X, d)$ satisfying $\htop(T)<+\infty$ and
$\phi: \mathcal{O}_2(T) \rightarrow \mathbb{R}$ be a continuous function. Then
\begin{enumerate}
  \item[(i)] $e^1_\phi(X,T)\subset t_\phi(X,T)$.
  \item[(ii)] $e^2_\phi(X,T)= t_\phi(X,T)$.
\end{enumerate}
\end{thm}

\begin{proof}
Let $(\mathcal{Q},\mu)\in e^1_\phi(X,T)$. By \cite[Theorem D]{LLZ23} and Lemma \ref{lem:a integral equality}, one has
\begin{align*}
\Ptop(T,\phi+\varphi)-\Ptop(T,\phi)&\geq h_\mu (\mathcal{Q}) + \int_X \int_{T(x_1)} (\phi+\varphi) (x_1, x_2) \, d\mathcal{Q}_{x_1}(x_2) \, d\mu (x_1)\\&-h_\mu (\mathcal{Q}) + \int_X \int_{T(x_1)} \phi (x_1, x_2) \, d\mathcal{Q}_{x_1}(x_2) \, d\mu (x_1)\\
&=\int_X \int_{T(x_1)} \varphi (x_1, x_2) \, d\mathcal{Q}_{x_1}(x_2) \, d\mu (x_1)\\
&=\int_{\mathcal{O}_2(T)}\varphi\, d(\mu\mathcal{Q}^{[1]})
\end{align*}
for any  $\varphi\in C(\mathcal{O}_2(T))$. So we have $(\mathcal{Q},\mu)\in t_\phi(X,T)$.

By combining Theorem \ref{thm:type II variational principle} and a similar approach, we can get $e^2_\phi(X,T)\subset t_\phi(X,T)$. Given $(\mathcal{Q},\mu)\in t_\phi(X,T)$. From Lemma \ref{lem:property of tangent functional} we deduce that
\[
\Ptop(T,\phi)-\int_{\mathcal{O}_2(T)}\phi\, d(\mu\mathcal{Q}^{[1]})=\inf\left\{\Ptop(T,\psi)-\int_{\mathcal{O}_2(T)}\psi\, d(\mu\mathcal{Q}^{[1]}):\psi\in C(\mathcal{O}_2(T))\right\}.
\]
Therefore, applying Theorem \ref{thm:type II variational principle} and Lemma \ref{lem:a integral equality} gives
\begin{align*}
\mathfrak{h}_\mu (\mathcal{Q}) &= \inf_{\psi\in C(\mathcal{O}_2(T))}
\left\{ \Ptop(T,\psi) - \int_X \int_{T(x_1)} \psi (x_1, x_2) \, d\mathcal{Q}_{x_1}(x_2) \, d\mu (x_1) \right\}\\
&=\inf\left\{\Ptop(T,\psi)-\int_{\mathcal{O}_2(T)}\psi\, d(\mu\mathcal{Q}^{[1]}):\psi\in C(\mathcal{O}_2(T))\right\}\\
&=\Ptop(T,\phi)-\int_{\mathcal{O}_2(T)}\phi\, d(\mu\mathcal{Q}^{[1]}).
\end{align*}
So $(\mathcal{Q},\mu)\in e^2_\phi(X,T)$.
\end{proof}

\begin{cor}
Let $T$ be a correspondence on a compact metric space $(X, d)$ satisfying $\htop(T)<+\infty$ and
$\phi: \mathcal{O}_2(T) \rightarrow \mathbb{R}$ be a continuous function. Then there exists a dense subset of $C(\mathcal{O}_2(T))$ such that each member of this subset has a unique type II equilibrium state and has at most
one type I equilibrium state. This uniqueness carries the same meaning as described in Theorem \ref{thm:differentiability of pressure}.
\end{cor}
\begin{proof}
Since a convex function on a separated Banach space has a unique tangent functional at a dense set of points. Combining Theorem \ref{thm:characterization of equilibrium states} and the proof of Theorem \ref{thm:differentiability of pressure}, we obtain the result.
\end{proof}

According to Theorem \ref{thm:differentiability of pressure},
if one of the following two conditions holds:
\begin{enumerate}
  \item[(i)] $T$ is a forward expansive correspondence with the specification property;
  \item[(ii)] $T$ is an open, strongly transitive, distance-expanding correspondence on $X$,
\end{enumerate}
then the topological pressure of $T$ is Gateaux differentiable at
$\phi\equiv 0$ and the set $t_0(X,T)$ is unique. By \cite[Theorem B]{LLZ23}, \cite[Theorem C]{LLZ23}, and Theorem \ref{thm:characterization of equilibrium states}, we have
$e^1_0(X,T)=t_0(X,T)=e^2_0(X,T)$. So we have
$h_\mu (\mathcal{Q})=\mathfrak{h}_\mu (\mathcal{Q})=\htop(T)$ for $(\mathcal{Q},\mu)\in t_0(X,T)$. However, the following question is open.

\begin{ques}
If $T$ is a correspondence on a compact metric space $(X, d)$ satisfying either of the above two conditions, does it hold that
\[
h_\mu (\mathcal{Q})=\mathfrak{h}_\mu (\mathcal{Q})
\]
for any transition probability kernel \(\mathcal{Q}\) on \((X, \mathscr{B}(X))\) and any \(\mathcal{Q}\)-invariant Borel probability measure $\mu$ on $(X,\mathscr{B}(X))$.
\end{ques}


\begin{thebibliography}{99}

\bibitem{AK21}
L. Alvin and J. Kelly,
Topological entropy of Markov set-valued functions,
Ergodic Theory Dyn. Syst., 41 (2021), 321--337.


\bibitem{AFL91}
J. Aubin, H. Frankowska and A. Lasota,
Poincar\'{e}'s recurrence theorem for set-valued dynamical systems,
Ann. Polon. Math., 54 (1991), 85--91.



\bibitem{BCMV22}
A. Bi\'{s}, M. Carvalho, M. Mendes and P. Varandas,
A Convex Analysis Approach to Entropy Functions, Variational Principles and Equilibrium States,
Commun. Math. Phys., 394 (2022), 215--256.



\bibitem{CPMP08}
A. Chinchuluun, P. Pardalos, A. Migdalas and L. Pitsoulis,
Pareto Optimality, Game Theory and Equilibria, Springer, New York, 2008.



\bibitem{Co18-1}
F. Colonius,
Invariance entropy, quasi-stationary measures and control sets,
Discrete Contin. Dyn. Syst., 38 (2018), 2093--2123.



\bibitem{Co18-2}
F. Colonius,
Metric invariance entropy and conditionally invariant measures,
Ergodic Theory Dynam. Systems, 38 (2018), 921--939.




\bibitem{CCS18}
F. Colonius, A.J.N. Cossich and A.J. Santana,
Invariance pressure of control sets,
SIAM J. Control Optim., 56 (2018), 4130--4147.




\bibitem{CCS20}
F. Colonius, A.J.N. Cossich and A.J. Santana,
Bounds for invariance pressure,
J. Differ. Equ., 268 (2020), 7877--7896.



\bibitem{CCS22}
F. Colonius, J.A.N. Cossich and A. Santana,
Controllability Properties and Invariance Pressure for Linear Discrete-Time Systems,
J. Dyn. Differ. Equ., 34 (2022), 5--28.



\bibitem{CK08}
F. Colonius and C. Kawan,
Invariance Entropy for Control Systems,
SIAM J.Control Optim., 48 (2008), 1701--1721.


\bibitem{CSC19}
F. Colonius, A.J. Santana and A.J.N. Cossich,
Invariance pressure for control systems,
J. Dyn. Differ. Equ., 31 (2019), 1--23.



\bibitem{CP16}
W. Cordeiro and M. Pac\'{i}fico,
Continuum-wise expansiveness and specfication for
set-valued functions and topological entropy,
Proc. Amer. Math. Soc., 144 (2016), 4261--4271.




\bibitem{Ga16}
J. Gall,
Brownian motion, martingales, and stochastic calculus,
Springer, Switzerland, 2016.



\bibitem{IM06}
W. Ingram and W. Mahavier,
Inverse limits of upper semi-continuous set valued functions,
Houston J. Math., 32 (2006), 119--130.




\bibitem{KH95}
A. Katok and B. Hasselblatt,
Introduction to the Modem Theory of Dynamical Systems.
Encyclopedia of Mathematics and its Applications,
vol. 54, Cambridge University Press, London-New York, 1995.



\bibitem{Ka13}
C. Kawan,
Invariance Entropy for Deterministic Control Systems,
Lecture Notes in Mathematics 2089, 2013.



\bibitem{Ke98}
G. Keller,
Equilibrium States in Ergodic Theory,
London Mathematical Society Student Texts 42,
Cambridge University Press, 1998.



\bibitem{KT17}
J. Kelly and T. Tennant,
Topological entropy of set-valued functions,
Houston J. Math., 43 (2017), 263--282.


\bibitem{LLZ23}
X. Li, Z. Li and Y. Zhang,
Thermodynamic formalism for correspondences,
arXiv preprint arXiv:2311.09397, 2023.



\bibitem{LMM23}
M. Yu. Lyubich, J. Mazor and S. Mukherjee,
Antiholomorphic correspondences and mating I: realization theorems,
Preprint, (arXiv:2303.02459), 2023.



\bibitem{MT12}
S. Meyn and R. Tweedie,
Markov chains and stochastic stability, Springer, London, 2012.



\bibitem{Mi95}
W. Miller,
Frobenius-Perron operators and approximation of invariant measures for set-valued dynamical systems,
Set-Valued Anal., 3 (1995), 181--194.



\bibitem{MA99}
W. Miller and E. Akin,
Invariant measures for set-valued dynamical systems,
Trans. Amer. Math. Soc., 351 (1999), 1203--1225.



\bibitem{NEMM04}
G. Nair, R. Evans, I. Mareels and W. Moran,
Topological feedback entropy and Nonlinear stabilization,
IEEE Trans. Automat. Control, 49 (2004), 1585--1597.



\bibitem{NWH22}
X. Nie, T. Wang and Y. Huang,
Measure-theoretic invariance entropy and variational principles for control systems,
J. Differ. Equ., 321 (2022), 318--348.


\bibitem{PV17}
M. Pacifico and J. Vieitez,
Expansiveness, Lyapunov exponents and entropy for set valued maps,
Preprint, (arXiv:1709.05739), 2017.



\bibitem{Pe97}
Ya. Pesin,
Dimension Theory in Tynamical Systems: Contemporary views and applications,
University of Chicago Press, Chicago, 1997.



\bibitem{Pe93}
L. Petrosyan,
Differential Games of Pursuit, World Scientfic. 1993.



\bibitem{Po21}
J. Pommaret,
Differential Correspondences and Control Theory,
Advances in Pure Math., 11 (2021), 835--882.



\bibitem{RT18}
E. Raines and T. Tennant,
The specification property on a set-valued map and its inverse limit,
Houston J. Math., 44 (2018), 665--677.



\bibitem{Ru73}
D. Ruelle,
Statistical mechanics on a compact set with $Z^v$ action satisfying expansiveness and specification,
Trans. Amer. Math. Soc., 185 (1973), 237--251.



\bibitem{Wa75}
P. Walters,
A variational principle for the pressure of continuous transformations,
Amer. J. Math., 17 (1975), 937--971.


\bibitem{Wa82}
P. Walters,
An Introduction to Ergodic Theory,
Springer-Verlag, New York--Berlin, 1982.



\bibitem{Wa92}
P. Walters,
Differentiability properties of the pressure of a continuous transformation on a compact metric space,
J. Lond. Math. Soc., 46 (1992), 471--481.




\bibitem{WHS19}
T. Wang, Y. Huang and H.W. Sun,
Measure-theoretic invariance entropy for control systems,
SIAM J. Control Optim., 57 (2019), 310--333.


\bibitem{WH22}
T. Wang and Y. Huang,
Inverse variational principles for control systems,
Nonlinearity, 35 (2022), 1610--1633.



\bibitem{WZZ23}
X. Wang, Y. Zhang and Y. Zhu,
On various entropies of set-valued maps,
J. Math. Anal. Appl., 524 (2023), 127097.



\bibitem{ZZ24}
Y. Zhang and Y. Zhu,
Topological stability and entropy for certain set-valued maps,
Acta Math. Sin. Engl. Ser., 40 (2024), 962--984.



\bibitem{ZH19}
X. Zhong and Y. Huang,
Invariance pressure dimensions for control systems,
J. Dyn. Differ. Equ., 31 (2019), 2205--2222.


\end{thebibliography}
\end{document}